\documentclass[12pt]{article}

\usepackage{graphics,amsthm, amsmath, amssymb}

\usepackage{color}

\usepackage{graphics}
\usepackage{graphicx}

\usepackage{multicol}

\expandafter\chardef\csname pre amssym.def at\endcsname=\the\catcode`\@
\catcode`\@=11

\def\undefine#1{\let#1\undefined}
\def\newsymbol#1#2#3#4#5{\let\next@\relax
 \ifnum#2=\@ne\let\next@\msafam@\else
 \ifnum#2=\tw@\let\next@\msbfam@\fi\fi
 \mathchardef#1="#3\next@#4#5}
\def\mathhexbox@#1#2#3{\relax
 \ifmmode\mathpalette{}{\m@th\mathchar"#1#2#3}%
 \else\leavevmode\hbox{$\m@th\mathchar"#1#2#3$}\fi}
\def\hexnumber@#1{\ifcase#1 0\or 1\or 2\or 3\or 4\or 5\or 6\or 7\or 8\or
 9\or A\or B\or C\or D\or E\or F\fi}

\font\tenmsa=msam10
\font\sevenmsa=msam7
\font\fivemsa=msam5
\newfam\msafam
\textfont\msafam=\tenmsa
\scriptfont\msafam=\sevenmsa
\scriptscriptfont\msafam=\fivemsa
\edef\msafam@{\hexnumber@\msafam}
\mathchardef\dabar@"0\msafam@39
\def\dashrightarrow{\mathrel{\dabar@\dabar@\mathchar"0\msafam@4B}}
\def\dashleftarrow{\mathrel{\mathchar"0\msafam@4C\dabar@\dabar@}}

\font\tenmsb=msbm10
\font\sevenmsb=msbm7
\font\fivemsb=msbm5
\newfam\msbfam
\textfont\msbfam=\tenmsb
\scriptfont\msbfam=\sevenmsb
\scriptscriptfont\msbfam=\fivemsb
\edef\msbfam@{\hexnumber@\msbfam}

\theoremstyle{plain}
\newtheorem{theorem}{Theorem}[section]
\newtheorem{lemma}[theorem]{Lemma}

\newtheorem{corollary}[theorem]{Corollary}

\theoremstyle{definition}
\newtheorem{definition}[theorem]{Definition}
\newtheorem{example}[theorem]{Example}

\newtheorem{remark}[theorem]{Remark}

\def\Resultant{{\rm Res}}
\def\res{{\rm Res}}
\def\deg{{\rm deg}}
\def\degree{{\rm deg}}

\def\cP{{\mathcal P}}
\def\cPx{{\mathcal P}_{x_1}}
\def\cPy{{\mathcal P}_{x_2}}
\def\cPz{{\mathcal P}_{x_3}}
\def\cPw{{\mathcal P}_{x_4}}

\def\cpp{{\mathfrak p}}

\def\cQ{{\mathcal Q}}
\def\cH{{\mathcal H}}

\def\cc{{\mathcal C}}

\def\cS{{\mathcal S}}
\def\cZ{{\mathcal Z}}
\def\cA{{\mathcal A}^\star}

\def\cX{{\rm X}}
\def\cY{{\rm Y}}

\def\cSx{{\mathcal S}_{x_1}}
\def\cSy{{\mathcal S}_{x_2}}
\def\cSz{{\mathcal S}_{x_3}}
\def\cSw{{\mathcal S}_{x_4}}
\def\cSm{{\mathfrak S}}
\def\cF{{\mathfrak F}}
\def\cB{{\mathfrak B}}

\def\qed{\hfill  \framebox(5,5){}}

\def\para{\vspace{2 mm}}

\def\card{{\rm Card}}
\def\mult{{\rm mult}}
\def\lcm{{\rm lcm}}

\def\K{{\mathbb K}}

\def\projdos{{\mathbb P}^2(\K)}
\def\projtres{{\mathbb P}^3(\K)}

\def\tt{{t^*}}
\def\ss{{s^*}}

\def\ttt{{\mathfrak t}}
\def\sss{{\mathfrak s}}
\def\vvv{{\mathfrak v}}

\def\ggg{{\mathfrak g}}

\def\tv{{\mathbf t}}
\def\sv{{\mathbf s}}

\def\numer{{\rm Numer}}
\def\pp{{\rm PrimPart}}
\def\content{{\rm Content}}
\def\ii{{\,\imath}\,}

\begin{document}

\title{Computing the Singularities of Rational Surfaces\thanks{This work was developed, and partially supported, under the research project MTM2008-04699-C03-01 "Variedades paramétricas: algoritmos y aplicaciones", Ministerio de Ciencia e Innovación, Spain
and by " Fondos Europeos de Desarrollo Regional" of the European Union.}}

\author{
S. P\'erez--D\'{\i}az, J.R. Sendra, C. Villarino  \\
Dpto. de Matem\'aticas \\
        Universidad de Alcal\'a \\
      E-28871 Madrid, Spain  \\
$\{$sonia.perez,rafael.sendra,carlos.villarino$\}$@uah.es
}
\date{}          
\maketitle

\begin{abstract}
 Given a rational projective parametrization $\cP(\ttt,\sss,\vvv)$ of a rational projective surface $\cS$ we present an algorithm such that, with the exception of a finite set (maybe empty) $\cB$ of projective base points of $\cP$, decomposes the projective parameter plane as $\projdos\setminus \cB=\cup_{k=1}^{\ell} \cSm_k$ such that if $(\ttt_0:\sss_0:\vvv_0)\in \cSm_k$ then $\cP(\ttt_0,\sss_0,\vvv_0)$ is a point of $\cS$ of multiplicity $k$.
\end{abstract}

\section{Introduction}\label{sec-intro}

The study, analysis and computation of the singular locus of algebraic varieties is an old but still very active research topic.   The interest on the study of  singularities is motivated by multiple reasons, being one of them their applicability; for instance, in geometric modeling, when determining
the shape and the topology of curves (either planar or spatial) and of surfaces, etc.  In this paper, we focus on the problem of computing the singularities, as well as their multiplicities, of rational surfaces given parametricaly.

\para

When the algebraic variety is given as a zero set of finitely many polynomials, the singularities  and their  multiplicities can be directly computed by applying elimination theory techniques as Gr\"obner bases, characteristic sets, etc. However when the algebraic variety is unirational and it is given by means of a rational parametrization, besides the question of computing the singular locus and its multiplicity structure, one has the additional problem of determining the parameter values that generate the singular points with their corresponding multiplicities. This, for instance, can be useful when using a parametrization for plotting a curve or a surface or when utilizing a parametrization for analyzing the intersection variety of two varieties being one of them  parametrically given.
 Of course, one can always apply elimination techniques to first provide the defining implicit polynomials of the variety, second to  determine the singularities from these polynomials, third to decompose the singular locus w.r.t. the multiplicities,  and finally   to compute the fibre (w.r.t. the parametrization) of the elements in the singular locus. Nevertheless, this can be  inefficient because of the computational complexity.

 \para

 So the challenge, in the unirational case, is to derive the singularities  and their multiplicity directly from a parametric representation avoiding the computation of the  ideal of the variety. The case of rational curves (both planar and spatial) has  been addressed by several authors (see \cite{chen}, \cite{Park}, \cite{Sonia-sing}, \cite{Rubio}). However, the case of rational surfaces
has not been so extensively studied. We refer  the reader to \cite{chena-regladas} where the case of rational ruled surfaces is analyzed.

\para

In this paper, we present an algorithm for computing the singularities of a rational projective surface from an input rational projective  parametrization not necessarily proper (i.e., birational). More precisely, the problem we  deal with is stated as:

\para

\noindent \underline{\sf Problem statement}
\begin{itemize}
\item Given a rational projective parametrization $$\cP(\ttt,\sss,\vvv)=(\cpp_1(\ttt,\sss,\vvv):\cdots:\cpp_4(\ttt,\sss,\vvv)),$$  of a rational projective surface $\cS\subset \projtres$, where $\K$ is an algebraically closed field of characteristic zero.
\item Decompose $\projdos$ as $\projdos=\cup_{k=1}^{\ell} \cSm_k$ such that if $(\ttt_0:\sss_0:\vvv_0)\in \cSm_k$ then $\cP(\ttt_0,\sss_0,\vvv_0)$ is a point of $\cS$ of multiplicity $k$.
\end{itemize}
Although abusing the terminology, we will use the following definition

\para

\begin{definition}
The elements in $\cSm_1$ are called {\sf $\cP$-simple points of $\cS$}, and the elements in $\cSm_k$, with $k>1$, {\sf $\cP$-singularities of $\cS$ of multiplicity $k$}. We refer to these points as {\sf affine (either $\cP$-simple or $\cP$-singular) points} if  $\vvv_0\neq 0$ and {\sf points (either $\cP$-simple or $\cP$-singular) at infinity} if $\vvv_0=0$. Moreover, we represent the multiplicity of $(\ttt_0:\sss_0:\vvv_0)$ as $\mult((\ttt_0:\sss_0:\vvv_0))$ meaning $$\mult((\ttt_0:\sss_0:\vvv_0))=\mult(\cP(\ttt_0,\sss_0,\vvv_0),\cS))$$ where $\mult(A,\cS)$ denotes the multiplicity of $A\in \projtres$ w.r.t. $\cS$.
\end{definition}

\para

The  polynomials $\cpp_1,\ldots,\cpp_4$ are assumed to be homogeneous of the same degree and coprime. Therefore the parametrization $\cP(\ttt,\sss,\vvv)$ induces the regular map
\[ \begin{array}{rccc} \cP:&\projdos\setminus \cB &\rightarrow &
\cS \\ & \overline{\alpha} & \mapsto & \cP(\overline{\alpha}) \end{array} \]
where $\cB=\{ \overline{\alpha}\in \projdos \,|\, \cpp_1(\overline{\alpha})=\cdots=\cpp_4(\overline{\alpha})=0\}$;  we call the elements in $\cB$ the base points of the parametrization (see Section \ref{sec-notation}). We will be able to decompose, as above,  $\projdos\setminus \cB$.

\para

$\cB$ is either zero dimensional or empty. So, we will be missing (at most) finitely many parameter values in $\projdos$. On the other hand if $\cB=\emptyset$, since $\cS$ is irreducible and $\cP$ regular, then $\cP(\projdos)=\cS$ (see e.g. Theorem 2, page 57, in \cite{shafa}). Therefore, if $\cB=\emptyset$, our method will determine all singularities of $\cS$. However, if $\cB\neq \emptyset$ the method will generate all singularities in
the dense set $\cP(\projdos\setminus \cB) \subset \cS$. For avoiding this deficiency one may consider reparametrizing normally the parametrization, however this not an easy task (see \cite{normalidad}). We do not deal with this issue in this paper.

\para

Our method is based on the generalization of the ideas in \cite{Sonia-sing} in combination with the results in \cite{PS-grado} and \cite{PS-ISSAC} that perform the computations without implicitizing. Intuitively speaking, the method works as follows; see Section \ref{sec-notation} for further details. First we state a formula for computing the multiplicity of an affine point w.r.t. an affine surface (see Section \ref{sec-formula}). Then,  we analyze the multiplicity of the (affine) parameter values of the form $(\ttt_0:\sss_0:1)$ to later study the parameter values (at infinity) of the form $(\ttt_0:\sss_0:0)$. In order to compute $\mult((\ttt_0:\sss_0:1))$ we consider the four affine rational parametrizations (we call them $\cPx,\ldots,\cPw$) generated by $\cP(\ttt,\sss,\vvv)$ by dehomogenizing w.r.t.  the first, second, third and fourth component of the parametrization, respectively and taking $\vvv=1$. Then, we apply the multiplicity formula to $(\ttt_0:\sss_0:1)$ via $\cPw$. This first attempt will classify all affine parameter values  with the exception of a proper closed set, and hence with the exception of finitely many component of dimension either 1 or 0. By using consecutively $\cPz,\cPy$ and $\cPx$ we achieve the multiplicity of all affine parameter values not covered by $\cPw$ and not being base points (see Section \ref{sec-sing}). Finally we deal with the parameter values at infinity with a similar strategy but dehomogenizing with either $\ttt=1$ or $\sss=1$.

\para

The structure of the paper is as follows. In Section \ref{sec-notation} we introduce the notation as well as the general assumptions that essentially imposed that $\cS$ is not a plane. In Section \ref{sec-formula} we state the multiplicity formula, we develop a method for computing a point not on the surface (this will be needed in the algorithm), starting from the parametric representation and without implicitizing. Moreover, we briefly recall some procedures from \cite{PS-grado} and \cite{PS-ISSAC}. Sections \ref{sec-sing} and \ref{sec-singular-infinity} deal, respectively, with the  affine $\cP$-singularities and the $\cP$-singularities at infinity. Section \ref{sec-algorithm} summarizes all the ideas by deriving an algorithm. Also, a complete example is given. Sections \ref{sec-sing} and \ref{sec-singular-infinity} contain the technicalities of the theoretical argumentation of our method. A reader not interested in that theoretical argumentation might skip these sections to go to Section  \ref{sec-algorithm} to directly apply the procedure.

\section{Notation, general assumptions and strategy}\label{sec-notation}

In this section we introduce the notation that will be used throughout the paper, as well as the general assumptions.

 \para

 $\K$ is an algebraically closed field of characteristic zero, and $\projdos, \projtres$ are the projective plane and projective space over $\K$, respectively. Let $(x_1:x_2:x_3:x_4)$ be the projective coordinates in $\projtres$

 \para

Let $\cS\subset \projtres$ be a  rational projective surface different   of the planes $x_i=0$. This is not a loss of generality because, in that case, the surface is smooth.
In addition, let $\cP(\ttt,\sss,\vvv)$ be a rational projective parametrization of $\cS$.  We consider that $\cP(\ttt,\sss,\vvv)$ is expressed as
\[ \cP(\ttt,\sss,\vvv)=\left(\cpp_1(\ttt,\sss,\vvv):\cpp_2(\ttt,\sss,\vvv):\cpp_3(\ttt,\sss,\vvv):\cpp_4(\ttt,\sss,\vvv)\right)\]
where $\gcd(\cpp_1,\cpp_2,\cpp_3,\cpp_4)=1$ and the four polynomials are homogeneous (note that none of them is zero) of the same degree. We say the $(\ttt_0:\sss_0:\vvv_0)\in \projdos$ is a {\sf (projective) base point} of $\cP(\ttt,\sss,\vvv)$ if $\cpp_1(\ttt_0,\sss_0,\vvv_0)=\cpp_2(\ttt_0,\sss_0,\vvv_0)=\cpp_3(\ttt_0,\sss_0,\vvv_0)=\cpp_4(\ttt_0,\sss_0,\vvv_0)=0$. We denote by $\cB$ the set of (projective) base points of $\cP(\ttt,\sss,\vvv)$. Since $\cB$ is the intersection of the projective curves defined by $\cpp_i(\ttt,\sss,\vvv)$, and since $\gcd(\cpp_1,\cpp_2,\cpp_3,\cpp_4)=1$, we get  the following lemma.

\para

\begin{lemma}\label{lemma-ptos-base-cardinal}
$\card(\cB)<\infty$.
 \end{lemma}

\para

Our strategy will reduce the problem to the affine case to afterwards  analyze the points at infinity. We  denote by $\cS_{x_i}$ the affine surface obtained by dehomogenization of  $\cS$ with $x_i=1$; note that $\cS$ is not the plane $x_i=0$. Also, we denote by $\cP_{x_i}(t,s)$  the corresponding affine parametrization obtained from $\cP(\ttt,\sss,\vvv)$.
More precisely (say $i=4$, and that $\gamma=\deg(\cpp_i)$), $\cP(\ttt,\sss,\vvv)$ can be replaced by (note that $\cpp_4\neq 0$)
\[\left(\frac{\cpp_1(\ttt,\sss,\vvv)}{\cpp_4(\ttt,\sss,\vvv)}:\frac{\cpp_2(\ttt,\sss,\vvv)}{\cpp_4(\ttt,\sss,\vvv)}:
\frac{\cpp_3(\ttt,\sss,\vvv)}{\cpp_4(\ttt,\sss,\vvv)}:1\right)= \]
\[
\left(\frac{\vvv^\gamma \cpp_1(\frac{\ttt}{\vvv},\frac{\sss}{\vvv},1)}{\vvv^\gamma \cpp_4(\frac{\ttt}{\vvv},\frac{\sss}{\vvv},1)}:\frac{\vvv^\gamma \cpp_2(\frac{\ttt}{\vvv},\frac{\sss}{\vvv},1)}{\vvv^\gamma \cpp_4(\frac{\ttt}{\vvv},\frac{\sss}{\vvv},1)}:\frac{\vvv^\gamma \cpp_3(\frac{\ttt}{\vvv},\frac{\sss}{\vvv},1)}{\vvv^\gamma \cpp_4(\frac{\ttt}{\vvv},\frac{\sss}{\vvv},1)}:1\right)=\left(\frac{\cpp_1(\frac{\ttt}{\vvv}:\frac{\sss}{\vvv}:1)}{\cpp_4(\frac{\ttt}{\vvv}:\frac{\sss}{\vvv},1)},
\frac{\cpp_2(\frac{\ttt}{\vvv},\frac{\sss}{\vvv},1)}{\cpp_4(\frac{\ttt}{\vvv},\frac{\sss}{\vvv},1)}:\frac{\cpp_3(\frac{\ttt}{\vvv},\frac{\sss}{\vvv},1)}{ \cpp_4(\frac{\ttt}{\vvv},\frac{\sss}{\vvv},1)}:1\right) \]
So, introducing the notation $\{t=\frac{\ttt}{\vvv},s=\frac{\sss}{\vvv}\}$ and $\tilde{p}_i(t,s)=\cpp_i(t,s,1)$, we have that
\begin{multicols}{2}
\begin{itemize}
\item[] $\cPx(t,s)=\left(\dfrac{\tilde{p}_2(t,s)}{\tilde{p}_1(t,s)},\dfrac{{\tilde{p}}_3(t,s)}{\tilde{p}_1(t,s)},\dfrac{\tilde{p}_4(t,s)}{\tilde{p}_1(t,s)}  \right)$,
\item[] $\cPy(t,s)=\left(\dfrac{\tilde{p}_1(t,s)}{\tilde{p}_2(t,s)},\dfrac{\tilde{p}_3(t,s)}{\tilde{p}_2(t,s)},\dfrac{\tilde{p}_4(t,s)}{\tilde{p}_2(t,s)}  \right)$,
\item[] $\cPz(t,s)=\left(\dfrac{\tilde{p}_1(t,s)}{\tilde{p}_3(t,s)},\dfrac{\tilde{p}_2(t,s)}{\tilde{p}_3(t,s)},\dfrac{\tilde{p}_4(t,s)}{\tilde{p}_3(t,s)}  \right)$,
\item[] $\cPw(t,s)=\left(\dfrac{\tilde{p}_1(t,s)}{\tilde{p}_4(t,s)},\dfrac{\tilde{p}_2(t,s)}{\tilde{p}_4(t,s)},\dfrac{\tilde{p}_3(t,s)}{\tilde{p}_4(t,s)}  \right)$.
 \end{itemize}
\end{multicols}
\noindent  Observe that, since $\gcd(\cpp_1,\cpp_2,\cpp_3,\cpp_4)=1$, then $\gcd(\tilde{p}_1,\tilde{p}_2,\tilde{p}_3,\tilde{p}_4)=1$. Furthermore,
note that $\cP_{x_i}(t,s)$ is a rational parametrization of the affine rational surface $\cS_{x_i}$.
Analogously, we say that $(t_0,s_0)\in \K^2$ is an {\sf (affine) base point} of $\cP(\ttt,\sss,\vvv)$ if $\cpp_1(t_0,s_0,1)=\cpp_2(t_0,s_0,1)=\cpp_3(t_0,s_0,1)=\cpp_4(t_0,s_0,1)=0$. Let us denote by $\cB_a$ the set of affine base points of $\cP(\ttt,\sss,\vvv)$. Observe that $\cB_a$ can be  naturally embedded in $\cB$.

\para

Furthermore, we will consider that the rational functions in $\cP_{x_i}(t,s)$ are expressed in reduced form. So, in the sequel, we also use the following notation
\begin{multicols}{2}
\begin{itemize}
\item[] $\cPx(t,s)=\left(\dfrac{{p}_2(t,s)}{{q}_{1,1}(t,s)},\dfrac{{{p}}_3(t,s)}{{q}_{1,2}(t,s)},\dfrac{{p}_4(t,s)}{{q}_{1,3}(t,s)}  \right)$,
\item[] $\cPy(t,s)=\left(\dfrac{{p}_1(t,s)}{{q}_{2,1}(t,s)},\dfrac{{p}_3(t,s)}{{q}_{2,2}(t,s)},\dfrac{{p}_4(t,s)}{{q}_{2,3}(t,s)}  \right)$,
\item[] $\cPz(t,s)=\left(\dfrac{{p}_1(t,s)}{{q}_{3,1}(t,s)},\dfrac{{p}_2(t,s)}{{q}_{3,2}(t,s)},\dfrac{{p}_4(t,s)}{{q}_{3,3}(t,s)}  \right)$,
\item[] $\cPw(t,s)=\left(\dfrac{{p}_1(t,s)}{{q}_{4,1}(t,s)},\dfrac{{p}_2(t,s)}{{q}_{4,2}(t,s)},\dfrac{{p}_3(t,s)}{{q}_{4,3}(t,s)}  \right)$.
\end{itemize}
\end{multicols}
\noindent where all rational functions are in reduced form. As shown in the next lemma, the lcm of the denominators generates $\cB_a$.

\para

\begin{lemma}\label{lemma-ptos-base-afines}
$\cB_a=\bigcap_{i=1}^{4} \{(t_0,s_0)\in \K^2\,|\,\lcm(q_{i,1},q_{i,2},q_{i,3})=0\}$.
\end{lemma}
\begin{proof}  Let $\Delta_i:=\lcm(q_{i,1},q_{i,2},q_{i,3})$. We prove that $\Delta_i=\tilde{p}_i$, and from there the proof is trivial. Let us assume that $\Delta_1\neq \tilde{p}_1$; similarly for the others. Clearly $\Delta_1$ divides $\tilde{p}_1$. So, there exists a non-trivial factor $H_1$ such that $\tilde{p}_1=\Delta_1 H_1$. Therefore, by construction, $H_1$ divides $\tilde{p}_k$, with $k=2,3,4$. But this is a contradiction, since $\gcd(\tilde{p}_1,\ldots,\tilde{p}_4)=1$.
\end{proof}

\para

Furthermore, if $\Theta:\K^n \rightarrow \K^m$ is a rational affine map, we denote by $\deg(\Theta)$ the degree of the map $\Theta$ (see e.g.
\cite{shafa} pp.143, or \cite{Harris:algebraic} pp.80). In particular, $\deg(\cP_{x_i})$ denotes the degree of the rational map $\cP_{x_i}$ induced by the rational parametrization $\cP_{x_i}(t,s)$.

\para

Also, for a rational function $\chi$ we denote by $\numer(\chi)$ the numerator of $\chi$ when expressed in reduced form. By $\pp_{\{v_1,\ldots,v_n\}}(f)$ and $\content_{\{v_1,\ldots,v_n\}}(f)$, where $f\in \K[x_1,\ldots,x_m][v_1,\ldots,v_n]$, we denote the primitive part and the content w.r.t. $\{v_1,\ldots,v_n\}$ of $f$, respectively. For $f,g$ polynomials depending on $s$ we denote $\res_s(f,g)$ its resultant.

\para

For $P\in \projtres$, we represent by $\mult(P,\cS)$ the multiplicity of $P$ on $\cS$; note that if $P$ can be seen in the same affine space as $\cS_{x_i}$ then $\mult(P,\cS)=\mult(P,\cS_{x_i})$.

\para

\begin{center}
{\sf General assumption}
\end{center}

\para

We   assume that for every two different polynomials
$\cpp_i,\cpp_j$ it does not exist $\lambda\in \K$ such that $\cpp_i=\lambda \cpp_j$. Note that if  $\cpp_i=\lambda \cpp_j$, then $\cS$ is the plane of equation $x_i-\lambda x_j=0$, and the problem is trivial. In addition, note that this requirement implies that none of the dehomogenizations $\cS_{x_i}$  is empty ($\cpp_i\neq 0$ and hence $\cS$ is not the plane $x_i=0$). Moreover, $\cS_{x_i}$ is not   a plane  parallel to any of the affine coordinate plane in $\K^3$. So this does not imply any loss of generality.

\para

\begin{center}
{\sf Strategy}
\end{center}

\para

We briefly describe here the  ideas of our strategy. The precise details on how to execute them will come in the subsequence sections.
The main steps in our strategy are as follows; we recall that our goal is to decompose $\projdos$ such that for $(t_0:s_0:v_0)\in \projdos$ we know whether $\cP(t_0,s_0,v_0)$ is singular or simple in $\cS$, and if it is singular we also want to determine its multiplicity.
\begin{enumerate}
\item First we analyze the parameter values of the form $(t_0:s_0:1)$. For that, we work with $\cPw(t, s)$ and we treat the problem in $\K^2$. At this stage, we will be able to give an answer for  $\K^2\setminus \{(t_0,s_0)\in \K^2\,|\,\lcm(q_{4,1},q_{4,2},q_{4,3})=0\}$. Repeating the process (not necessarily for the whole affine plane, but for those parameters values that are not yet under control) with $\cPz(t,s)$, and if necessary with $\cPy(t,s)$ and $\cPx(t,s)$ we will be able to give an answer for (see Lemma \ref{lemma-ptos-base-afines})
    $$\displaystyle{\K^2\setminus \bigcap_{i=1}^{4} \{(t_0,s_0)\in \K^2\,|\,\lcm(q_{i,1},q_{i,2},q_{i,3})=0\}=\K^2\setminus \cB_a.}$$
  \item We analyze the case of  $(0:1:0)$, checking first whether $(0:1:0)\in \cB$.
\item We analyze the case of the parameter values $(1:\lambda:0)$. First we find those $\lambda$ values generating base points. Afterwards, we study (under a suitable dehomogenization) the rest of the points.
\end{enumerate}

\section{The multiplicity formula}\label{sec-formula}

In this section we state a formula for computing the multiplicity of a point in $\K^3$ w.r.t. an affine rational surface in $\K^3$, when a rational parametrization (not necessarily proper) is provided. As a consequence, we give a criterion  for an affine point to be regular on the affine rational surface. In order to derive an algorithmic version of these results, we will recall some procedures in \cite{PS-grado},\cite{PS-ISSAC} and \cite{JSC08}, and we will
present a method for determining a point out of the surface without knowing the implicit equation.

\para

For that purpose, throughout this section,
$\cZ\subset \K^3$ is a rational affine surface and
\[ \cQ(t,s)=\left(\frac{N_1(t,s)}{D_1(t,s)},\frac{N_2(t,s)}{D_2(t,s)},\frac{N_3(t,s)}{D_3(t,s)} \right) \]
a rational parametrization (in reduced form) of $\cZ$;  we assume w.l.o.g. that   $\cZ$ is not a plane parallel to the coordinate planes of $\K^3$.

\para

For any $A=(a,b,c)$ point of $\K^3$, let $f(x,y,z)$ be the defining polynomial of $\cZ$ and $F(x,y,z,w)$ its homogenization.  We consider the polynomial $g(x,y,z)=f(x+a,y+b,z+c)$, as well as $G(x,y,z,w)=F(x+aw,y+bw,z+cw,w)$. It is clear that
\[ \mult(A,\cZ)=\deg(G)-\deg_{w}(G). \]
On the other hand, note that
\[ \left(\frac{N_1(t,s)}{D_1(t,s)}-a,\frac{N_2(t,s)}{D_2(t,s)}-b,\frac{N_3(t,s)}{D_3(t,s)}-c,1 \right) \]
parametrizes the projective surface defined by $G$. Therefore, since we have assumed that $\cZ$ is not a plane parallel to the coordinate planes, then $N_1/D_1-a\neq 0$ and hence
\[ \cQ^*(t,s)=\left(\dfrac{N_2(t,s)-b D_2(t,s)}{N_1(t,s)-a D_1(t,s)} \cdot\dfrac{D_1(t,s)}{D_2(t,s)}, \dfrac{N_3(t,s)-c D_3(t,s)}{N_1(t,s)-a D_1(t,s)}\cdot \dfrac{D_1(t,s)}{D_3(t,s)}, \frac{D_1(t,s)}{N_1(t,s)-a D_1(t,s)} \right) \]
parametrizes the affine surface defined by $G(1,y,z,w)$; note that, since $G$ is homogeneous, $\deg_w(G)=\deg_w(G(1,y,z,w))$. Let us introduce the following notation
\[\Phi_{2,3}(A)(t,s)=\left(\dfrac{N_2(t,s)-b D_2(t,s)}{N_1(t,s)-a D_1(t,s)} \cdot\dfrac{D_1(t,s)}{D_2(t,s)}, \dfrac{N_3(t,s)-c D_3(t,s)}{N_1(t,s)-a D_1(t,s)}\cdot \dfrac{D_1(t,s)}{D_3(t,s)} \right) \]
and let $\Phi_{2,3}(A):  \K^2 \rightarrow  \K^2 $  be the induced map. Moreover, if for $i=1,2$, $\chi_{i}^{A}(t,s)$ denotes the $i$-component of $\Phi_{2,3}(A)(t,s)$,
let $\ggg_{i}^{\cQ,A}$ be the polynomial
\[ \ggg_{i}^{\cQ,A}(t,s,h_1,h_2)=\numer(\chi_{i}^{A}(t,s)-\chi_{i}^{A}(h_1,h_2)),\,\,\,i=1,2\]
where $h_1,h_2$ are new variables, and let $K(t,s,h_1,h_2)=\gcd(\ggg_{1}^{\cQ,A},\ggg_{2}^{\cQ,A})$ where the  gcd in computed in $\K[h_1,h_2][t,s]$. Then we introduce the polynomial
\[\ggg^{\cQ,A}(t,s,h_1,h_2)=\left\{\begin{array}{lll}
K(t,s,h_1,h_2) & &\mbox{if $\deg_{\{t,s\}}(K)>0$} \\
1 & & \mbox{if  $\deg_{\{t,s\}}(K)=0$} \end{array}\right. \]
\begin{remark}\label{remark-jac}
We observe that if $\ggg^{\cQ,A}=1$ then the determinant of the Jacobian of $\Phi_{2,3}(A)(t,s)$ is not identically zero
(see the preliminary paragraphs to Theorem 1 in \cite{PS-ISSAC}).
\end{remark}

\para
\noindent
In the following theorem and corollaries we assume that:
\begin{enumerate}
\item  none of the projective curves defined by each of the non-constant polynomials in $\{N_1,N_2,N_3,D_1,D_2,D_3\}$ passes through $(0:1:0)$
\item for each $A=(a,b,c)\in \K^3$ (similarly $A_0$)  none of the projective curves defined by the each of the non-constant polynomials in $\{N_2-bD_2,N_3-cD_3,N_1-aD_1\}$ passes through $(0:1:0)$.
  \end{enumerate}
Note that, if necessary, one can always perform a suitable polynomial linear change of parameters.

\para

\begin{theorem}\label{th-formula} {\sf (The general formula)}
It holds that
\begin{enumerate}
\item
$\mult(A,\cZ)=\deg(\cZ)$ iff $\ggg^{\cQ,A}\neq 1$.
\item $\mult(A,\cZ)<\deg(\cZ)$ iff $\ggg^{\cQ,A}=1$. Furthermore, if $\ggg^{\cQ,A}=1$
  then
 $$\deg(\cZ)-\mult(A,\cZ)=\dfrac{\degree(\Phi_{23}(A))}{\degree(\cQ)}, $$
\end{enumerate}
\end{theorem}
\begin{proof} (1) By Theorem 4 in \cite{PS-ISSAC}, $\deg_w(G)=0$ iff $\ggg^{\cQ,A}\neq 1$. Hence, $\mult(A,\cZ)=\deg(\cZ)$ iff $\ggg^{\cQ,A}\neq 1$.
\\
\noindent (2) All hypothesis of
Theorem 6 in \cite{JSC08} are satisfied. Therefore, $\deg_w(G)=\frac{\degree(\Phi_{23}(A))}{\degree(\cQ^*)}$, and the result follows by taking into account that $\deg(\cQ)=\deg(\cQ^*)$.
 \end{proof}

\begin{remark}\label{rem-A0} Note that:
\begin{enumerate}
\item  If $\ggg^{\cQ,A}\neq 1$, by Theorem \ref{th-formula} (1), then $A\in \cZ$.
\item For $A\not\in \cZ$, {$\deg(\Phi_{2,3}(A))$} is invariant, indeed it is $\deg(\cZ) {\degree(\cQ)}$.
\item Let $A=(a,b,c)$ and assume that $a\neq 0$; similarly for $b$ and $c$. We embed $A$ in $\projtres$ as $(a:b:c:1)$. Let $\overline{\cZ}$ be
 the projective closure of $\cZ$, and let $\overline{\cZ}_{x_1}$ the dehomogenization of $\overline{\cZ}$ w.r.t. $x_1=1$. Let $\cQ_{x_1}$ be the corresponding parametrization of $\overline{\cZ}_{x_1}$ generated by $\cQ$.
 Then $\deg(\cZ)=\deg(\overline{\cZ}_{x_1})$, $\mult((a,b,c),\cZ)=\mult((b/a,c/a,1/a),\overline{\cZ}_{x_1})$ and $\deg(\cQ)=\deg(\cQ_{x_1})$. Therefore, $\deg(\Phi_{2,3}((a,b,c)))=\deg(\Phi_{2,3}(b/a,c/a,1/a)))$.

\end{enumerate}
\end{remark}

\para

\begin{corollary}\label{cor-1-planos}
$\cZ$ is a plane if and only if  there exists a non-empty dense subset $\Omega$ of $\cZ$ such that for each $A\in \Omega$,  $\ggg^{\cQ,A}\neq 1$.
\end{corollary}
\begin{proof}
Since $\cZ$ is irreducible, the corollary follows from Theorem \ref{th-formula} (1).
\end{proof}

\para

\begin{corollary}\label{cor-conos}
Let $\cZ$ not be a plane. ${\mathcal Z}$ is a cone of vertex $A$ if and only if $\ggg^{\cQ,A}\neq 1$.
\end{corollary}
\begin{proof} Let ${\mathcal Q}(t,s)=t(g_1(s),g_2(s),g_3(s))$ be a rational parametrization of ${\mathcal Z}$ where we assume w.l.o.g. that $A$ is the origin $O$. The defining polynomial $f(x,y,z)$ of ${\mathcal Z}$ is a form  of degree $d>1$. Therefore $\mult(O,{\mathcal Z})=d=\deg({\mathcal Z})$ and by Theorem \ref{th-formula}, $\ggg^{\cQ,O}\neq 1$. Conversely, if $\ggg^{\cQ,A}\neq 1$ then $\deg_w(G)=0$ (see Theorem 4 in \cite{PS-ISSAC}). Thus, $G(x,y,z,w)=g(x,y,z)$ is an irreducible form of degree $d>1$. Let us see that $\{g(x,y,z)=0\}\cap\{x=1\}$ defines a rational plane curve $\cal D$ of degree $d$. Since $g$ is a form and $\cZ$ is not a plane, $\cal D$ is a curve of degree $d$. Moreover, since $\cZ$ is a surface, $\Phi_{2,3}(A)(t,s)$ is not constant  and parametrizes the surface defined by  $g(1,y,z)$. Then substituting either $t$ or $s$ by a suitable constant (say $t$ by $t_0$), ${\cal R}(s):=(1,\Phi_{2,3}(A)(t_0,s))$ parametrizes $\cal D$.

Now, ${\mathcal Q}(t,s)=A+t{\cal R}(s)$ defines a cone of vertex $A$ contained in ${\mathcal Z}$. Therefore, since ${\mathcal Z}$ is irreducible, it holds that  ${\mathcal Z}$ is the previous cone.
\end{proof}

\begin{corollary}\label{cor-conos-2}
Let $\cZ$ not be a plane. There exists at most one $A\in \K^3$ such that $\ggg^{\cQ,A}\neq 1$.
\end{corollary}
\begin{proof}
 Let us suppose that exist two points verifying the statement. We assume w.l.o.g., that one of them is the origin $O$ and the other one is ${A=(a,b,c)}\neq{O}$. By Corollary \ref{cor-conos}, $\cZ$ is a cone parametrized as ${\mathcal Q}(t,s)=t(g_1(s),g_2(s),g_3(s))$. Note that $f(x,y,z)$ is homogeneous of degree $d>1$.  Moreover $g(x,y,z)=f(x+a,y+b,z+c)$ can be expressed as $g(x,y,z)=f(x,y,z)+g^{\star}(x,y,z)$ where $\deg(g^{\star})<d$. By Theorem \ref{th-formula},  $\mult(A,{\mathcal Z})=\deg({\mathcal Z})$. Thus $f(x+a,y+b,z+c)=f(x,y,z)$. This implies that ${\mathcal Z}$ is invariant under the translation of the vector $A$. Let us see that ${\mathcal Z}$ is a plane which is a contradiction. Indeed, let $s_0\in{\mathbb K}$ be such that $v:=(g_1(s_0),g_2(s_0),g_3(s_0))$ is not parallel to $A$; observe that $s_0$ exists because $(g_1(s),g_2(s),g_3(s))$ is not a line. We consider the plane $\Pi$ given by the parametrization $p(s,t):=sA+vt$. Because of the invariance of $\cZ$, under translation of vector $A$, the family of lines $p(n,t)$ are included in $\Pi \cap  \cZ$, for $n\in {\mathbb N}\cup \{0\}$. Since all these lines are different and $\cZ$ is an irreducible algebraic set one concludes that $\cZ=\Pi$.
\end{proof}

\para

\para

\begin{corollary}\label{cor-1-formula} {\sf (The multiplicity formula)}
Let $A_0\in \K^3\setminus \cZ$ and let $A\in \K^3$.  Then
\begin{enumerate}
\item if $\ggg^{\cQ,A}=1$, then
 $$\mult(A,\cZ)=\dfrac{\degree(\Phi_{23}(A_0))-\degree(\Phi_{23}(A))}{\degree(\cQ)}, $$
\item if  $\ggg^{\cQ,A}\neq 1$, then
$$\mult(A,\cZ)=\dfrac{\degree(\Phi_{23}(A_0))}{\degree(\cQ)} $$
\end{enumerate}
\end{corollary}
\begin{proof} We prove (1); similarly for (2).
By Theorem \ref{th-formula}, one has that
$(\deg(\cZ)-\mult(A,\cZ))\degree(\cQ)=\degree(\Phi_{23}(A))$, and $\deg(\cZ)\degree(\cQ)=\degree(\Phi_{23}(A_0))$. From here the proof is obvious.
\end{proof}

\para

\begin{remark} Note that
\begin{enumerate}
\item  if there exists $A_0$ such that $\ggg^{\cQ,A_0}\neq 1$ (i.e., $\cZ$ is a cone), Theorem \ref{th-formula} and Corollary \ref{cor-1-formula} provide the degree of the surface.
    \item From Corollary \ref{cor-1-formula}, one deduces that $\degree(\Phi_{23}(A))$ is invariant for
    all the  $A\in \K^3$ having the same multiplicity w.r.t. $\cZ$.
    \end{enumerate}
\end{remark}

\para

The next corollary is a direct consequence of Corollary \ref{cor-1-formula}.

\para

\begin{corollary}\label{cor-2-formula} {\sf (Criterion for simple points)}
Let $A_0\in \K^3\setminus \cZ$ and let $A\in \K^3$. If $\cal Z$ is not a plane,  the following statements are equivalent
\begin{enumerate}
\item
$A$ is a simple point of $\cZ$.
\item  $\ggg^{\cQ,A}=1$ and $\degree(\Phi_{2,3}(A_0))-\degree(\Phi_{23}(A))=\degree(\cQ)$.
\end{enumerate}
\end{corollary}
\begin{proof}
By Theorem \ref{th-formula},
$(\deg(\cZ)-1)\degree(\cQ)=\degree(\Phi_{2,3}(A))$, and $\deg(\cZ)\degree(\cQ)=\degree(\Phi_{2,3}(A_0))$. From here the proof is obvious.
\end{proof}

We observe that if we know how to compute $\degree(\cQ)$, $\degree(\Phi_{23}(A))$ for any given $A\in \K^3$, and if we know how to compute a point out of the surface (recall that we do not have the implicit equation of $\cZ$), Corollary \ref{cor-1-formula} provides a method for computing the multiplicity of any point in $\K^3$, and Corollary \ref{cor-2-formula} a method   to check whether it is simple on  the surface.  We note that the $\deg(\cQ)$ is the {\sf index of improperness} of $\cQ(t,s)$; if $\cQ(t,s)$ is proper then this index is 1. Therefore, once the parametrization is given,  $\deg(\cQ)$ is fixed. However, $\degree(\Phi_{23}(A))$ will vary depending on $A$. Both quantities can be derived by applying elimination theory techniques as Gr\"obner basis. Indeed, they can be computed by means or resultants as shown in \cite{PS-grado} without determining the implicit equation of $\cQ$.

\para

In the following we  recall (as a recipe) how to compute $\degree(\cQ)$ and  $\degree(\Phi_{23}(A))$; for further details we refer to  \cite{PS-grado}, \cite{PS-ISSAC} or \cite{JSC08}. In addition,  we deduce a method for determining a point out of the surface.

\para

 \noindent \underline{\sf Method 1: Computation of $\degree(\cQ)$}

 \para

\noindent {\sf [Step 0]} Check the global hypothesis above:
 \vspace{2 mm}

 \parbox{15 cm}{{\sf [Step 0.1]} If any of the projective curves defined by the each of the non-constant polynomials in $\{N_1,N_2,N_3,D_1,D_2,D_3\}$ passes through $(0:1:0)$, apply a suitable (polynomial) linear change of parameters.}

  \parbox{15 cm}{{\sf [Step 0.2]} If the determinant of the Jacobian of $(\frac{N_2}{D_2},\frac{N_3}{D_3})$ is identically zero, apply a suitable linear change of coordinates in $\K^3$; namely, exchange suitably the affine coordinates in $\K^3$.}
\vspace{2 mm}

\noindent  {\sf [Step 1]} For $i=1,2,3$,  compute $G_i(t,s,h_1,h_2)=\numer\left(\frac{N_i(t,s)}{D_i(t,s)}-\frac{N_i(h_1,h_2)}{D_i(h_1,h_2)}\right)$.

\noindent  {\sf [Step 2]}  Determine $R(t,h_1,h_2,X)=\res_s(G_1,G_2+XG_3)$ where $X$ is a new variable.

\noindent {\sf [Step 3]}  Compute $S(t,h_1,h_2)=\pp_{\{h_1,h_2\}}(\content_{X}(R)))$.

\noindent {\sf [Step 4]}  $\degree(\cQ)=\deg_t(S)$.

\para

 \begin{center} {\sf Computation of  $\degree(\Phi_{23}(A))$} \end{center}

\para

We observe that by Theorem \ref{th-formula}, if $A_0\not\in \cZ$ then $\ggg^{\cQ,A_0}=1$. Therefore, by Corollary \ref{cor-1-formula}, we only need to compute $\degree(\Phi_{23}(A))$ for those $A$ such that  $\ggg^{\cQ,A}=1$; in particular when $A=A_0$.
Thus, in the following we assume that $A=(a,b,c)$ is such that $\ggg^{\cQ,A}=1$. Moreover, we will use the following technical lemma that will simplify the computations.

\para

In addition, since $\ggg^{\cQ,A}=1$, by Remark \ref{remark-jac}, the determinant of the Jacobian of $\Phi_{2,3}(A)(t,s)$ does not vanish. Therefore,
$\Phi_{2,3}(A)(\K^2)$ is dense in $\K^2$. So, we can compute the degree by taking a generic element $(\cX_1,\cX_2)\in \K^2$ as it is done in \cite{JSC08}. More precisely, we have the following method.

\para

 \noindent \underline{\sf Method 2: Computation of  $\degree(\Phi_{23}(A))$}

\para

\noindent {\sf [Step 0]} Check the global hypothesis above: if any of the projective curves defined by the each of the non-constant polynomials in $\{N_1-aD_1,N_2-bD_2,N_3-cD_3,N_1,N_2,N_3,D_1,D_2,D_3\}$ passes through $(0:1:0)$, apply a suitable (polynomial) linear change of parameters.

\para

\noindent {\sf [Step 1]}
 We take the  components $\chi_{i}^{A}(t,s)$ of $\Phi_{2,3}(A)(t, s):=(\chi_{1}^{A}(t,s),\chi_{2}^{A}(t,s))$

\noindent {\sf [Step 2]} For $i=1,2$, let
$H_i(t,s,\cX_i)=\numer(\chi_{i}^{A}(t,s)-\cX_i)\in \K[t,s,\cX_i].$

\noindent {\sf [Step 3]} $R(t,\cX_1,\cX_2)=\res_s(H_1,H_2)\in \K[t,\cX_1,\cX_2]$

\noindent {\sf [Step 4]} $\degree(\Phi_{2,3}(A))=\deg_t(\pp_{\{\cX_1,\cX_2\}}(R))$.

 \para

 \begin{center}
 {\sf Computation of a point $A_0$ out of $\cZ$}
 \end{center}

\para

For our reasoning we need to know
 the partial degree,  w.r.t. one of the variables,  of the defining polynomial of $\cZ$. Say that $m$ is the partial degree w.r.t. the variable $x$ (below we show how to compute $m$). This means
that for almost all affine lines $\cal L$ of the type $\{y=\lambda,z=\mu\}$  (recall that $\cZ$ is not a plane parallel to the coordinate planes) it holds that $\card({\cal L}\cap \cZ)=m$.  Then, the idea is as follows. We take values for $(\lambda, \mu)$ till the number of different points on $\cZ$ generated by $\cQ(t,s)$ is $m$. Note that for a fixed $(\lambda, \mu)$, these points are
\[ {\cal W}(\lambda, \mu)=\left\{\left(\dfrac{N_1(t_0,s_0)}{D_1(t_0,s_0)},\lambda, \mu\right)\,\left|\, \begin{array}{l} D_2(t_0,s_0)\lambda-N_2(t_0,s_0)=0, \\ D_3(t_0,s_0)\mu-N_3(t_0,s_0)=0, \\
\lcm(D_1,D_2,D_3)(t_0,s_0)\neq 0
\end{array}
\right\} \right. \]
Once we have found a suitable $(\lambda, \mu)$, every point $(\alpha, \lambda, \mu)\not\in {\cal W}(\lambda, \mu)$ is not on $\cZ$. We finish this section showing how to compute $m$ (see details in   Theorem 6 in \cite{JSC08}).

\para

 \noindent \underline{\sf Method 3: Computation of the partial degree w.r.t. $x$}

\para

\noindent  {\sf [Step 1]} Apply Method 1 to compute $\deg(\cQ)$

\noindent {\sf [Step 2]} Let, for $i=2,3$,   $G_i(t,s,h_1,h_2)$ as in Step 1 of Method 1.

\noindent {\sf [Step 3]} Return  $\dfrac{1}{\deg(\cQ)} \deg_{t}(\pp_{\{h_1,h_2\}}(\res_s(G_2(t,s,h_1,h_2),G_3(t,s,h_1,h_2)))$.

\para

\begin{remark}
Note that the polynomials $G_i$ are obtained in Step 1 of Method 1, and therefore it might happen that Step 0 of Method 1 was required. In that case,
we would have performed a linear change in the parameters $\{t, s\}$, and/or an affine linear change of coordinates $\{x, y, z\}$ consisting  in a permutation of variables.  The first situation does not affect to the partial degree of the polynomial. However, the second can. Nevertheless, if this is the case,  we only need to work with the new variable (the one exchange with $x$) and the corresponding lines perpendicular to  its corresponding coordinate plane.
\end{remark}

\para

\noindent \underline{\sf Method 4: Computation of $A_0\in \K^3\setminus \cZ$}

\para

\noindent  {\sf [Step 1]} Apply Method 3 to compute the partial degree  $m$ of the defining polynomial of $\cZ$ w.r.t. $x$.

\noindent {\sf [Step 2]} Give values to $(\lambda, \mu)\in \K^2$ till  $\card({\cal W}(\lambda, \mu))=m$, then take $A_0:=(\alpha, \lambda, \mu)\in \K^3\setminus {\cal W}(\lambda, \mu)$.

\para

 \begin{center} {\sf Computation of  $\mult(A,\cZ)$} \end{center}

\para

We finish the section, putting together all the previous ideas for computing the multiplicity of $A\in \K^3$ w.r.t. to the rational affine surface $\cZ$, parametrized by $\cQ(t,s)$.

\para

\noindent \underline{\sf Method 5: Computation of $\mult(A,\cZ)$}

\para

\noindent  {\sf [Step 1]} Apply Method 4 to find a point $A_0\not\in \cZ$.

\noindent  {\sf [Step 2]} Compute $\ggg^{\cQ,A}$.

\noindent  {\sf [Step 3]} If $\ggg^{\cQ,A}\neq 1$ then

 {\sf [Step  3.1]} Apply Method 2 to compute $n_1:=\deg(\Phi_{2,3}(A_0))$.

 {\sf [Step  3.2]} Apply Method 1 to compute $n_2:=\deg(\cQ)$.

 {\sf  [Step  3.3]} Return $\frac{n_1}{n_2}$

\noindent  {\sf [Step 4]} If $\ggg^{\cQ,A}= 1$ then

 {\sf [Step  4.1]} Apply Method 2 to compute $m_1:=\deg(\Phi_{2,3}(A))$ and $n_1:=\deg(\Phi_{2,3}(A_0))$.

 {\sf [Step  4.2]} Apply Method 1 to compute $n_2:=\deg(\cQ)$.

 {\sf [Step  4.3]} Return $\frac{n_1-m_1}{n_2}$

\section{Computing the affine $\cP$-singularities}\label{sec-sing}

In this section we see $\K^2$ embedded in $\projdos$ by means of the natural map $${\mathbf j}:\K^2 \rightarrow \projdos, (t_0,s_0)\mapsto (t_0:s_0:1);$$ in this sense, as already commented in Section \ref{sec-intro}, we will be determining the affine $\cP$-singularities of $\cS$.

\para

 For this purpose,  let $\Delta_i:=\{(t_0,s_0)\,|\,\lcm(q_{i,1},q_{i,2},q_{i,3})(t_0,s_0)=0\}$ and $\cB_a$ be the set of base points of $\cPw(t,s)$.
Note that ${\mathbf j}(\cB_a)\subset \cB$. The basic idea consists in applying Method 5 to a generic point on $\cS$. For this purpose, we proceed as follows.

\para

 \noindent \underline{\sf First Level.} We decompose $\Lambda_1:=\K^2\setminus \Delta_4$ as
\[ \Lambda_1:=\cup_{k=1}^{\ell_1} \cF_{k}^{1} \]
 such that if $(t_0,s_0)\in \cF_{k}^{1}$ then $\cPw(t_0,s_0)$ is a point of $\cSw$ of multiplicity $k$.

\para

 \noindent \underline{\sf Second Level.} If $\Delta_4\setminus \cB_a\neq \emptyset$ we decompose $\Lambda_2:=\Delta_4\setminus \Delta_3$ as
 \[ \Lambda_2:=\cup_{k=1}^{\ell_2} \cF_{k}^{2} \]
 such that if $(t_0,s_0)\in \cF_{k}^{2}$ then $\cPz(t_0,s_0)$ is a point of $\cSz$ of multiplicity $k$.

\para

 \noindent \underline{\sf Third level.} If $\Lambda_2 \setminus \cB_a\neq \emptyset$ we decompose $\Lambda_3:=(\Delta_4 \cap \Delta_3)\setminus \Delta_2$ as
 \[ \Lambda_3=\cup_{k=1}^{\ell_3} \cF_{k}^{3} \]
 such that if $(t_0,s_0)\in \cF_{k}^{3}$ then $\cPy(t_0,s_0)$ is a point of $\cSy$ of multiplicity $k$.

\para

 \noindent \underline{\sf Fourth Level.} If  $\Lambda_3\setminus \cB_a\neq \emptyset$ we decompose $\Lambda_4:=(\Delta_4 \cap \Delta_3\cap \Delta_2)\setminus \Delta_1$ as
 \[ \Lambda_4=\cup_{k=1}^{\ell_4} \cF_{k}^{4} \]
 such that if $(t_0,s_0)\in \cF_{k}^{4}$ then $\cPx(t_0,s_0)$ is a point of $\cSx$ of multiplicity $k$.

 \para

 Note that at this point, $\Lambda_4\setminus \cB_a=\emptyset$. Moreover ${\mathbf j}(\cup_{i=1}^{4}\cF_{k}^{i})\subset \cSm_k$ (see Section \ref{sec-intro}).

\para

\begin{center}
{\sf First level}
\end{center}

\para

The strategy for this level is as follows. We determine a closed set $\Delta^*$ of $\K^2$ such that for every $(t_0,s_0)\in \Lambda_1\setminus \Delta^*$ then $\cPw(t_0,s_0)$ is simple on $\cSw$; note that $\Lambda_1\setminus \Delta^* \subset \cF_{1}^{1}$. Next we decompose $\Lambda_1\cap \Delta^*$ as
$$\Lambda_1\cap \Delta^*=\cup_{k=1}^{\ell_{k}} \Delta (k) $$
such that  if $(t_0,s_0)\in \Delta (k)$ then  $\mult(\cP(t_0,s_0,1),\cS)=\ell_{k}$. Note that $\Delta (k)\subset \cF_{\ell_{k}}^{1}$.

\para

 \begin{center}
 {\sf First level (Part I): computation of $\Delta^*$}
 \end{center}

\para

In order to compute $\Delta^*$ we will determine some closed sets $\{\Delta^{*}_{i}\}_{i=0,..,4}$  of $\K^2$ such that
$\Delta^*=\cup_{i=0}^{4}\Delta^{*}_{i}.$
 For that purpose,  we apply Method 5 in Section \ref{sec-formula} to a generic point of $\cPw(\Lambda_1)\subset \cSw$; namely $\cPw(\tt,\ss)$ where $\tt,\ss$ are treated as new variables. We assume that we have already computed a point $A_0 $  in $\K^3\setminus \cSw$ (see  Step 1 in Method 5 or see Method 4 in Section \ref{sec-formula}) as well as $\deg(\Phi_{2,3}(A_0))$ and $\deg(\cPw)$. For simplifying the notation, throughout this section we will denote the generic point $\cPw(\tt,\ss)$ by $\cA$.

\para

To perform Step 2 in Method 5, we consider $\Phi_{2,3}(\cA)(t,s)$ as well as its  rational function components; namely
\[ \chi_{1}^{\cA}(t,s,\tt,\ss)=\dfrac{q_{4,2}(\tt,\ss)p_2(t,s)-p_2(\tt,\ss)q_{4,2}(t,s)}{q_{4,1}(\tt,\ss)p_1(t,s)-p_{1}(\tt,\ss)q_{4,1}(t,s)}
\cdot\dfrac{q_{4,1}(t,s)}{q_{4,2}(t,s)}\cdot\dfrac{q_{4,1}(\tt,\ss)}{q_{4,2}(\tt,\ss)} , \]  \[\chi_{2}^{\cA}(t,s,\tt,\ss)=\dfrac{q_{4,3}(\tt,\ss)p_3(t,s)-p_3(\tt,\ss)q_{4,3}(t,s)}{q_{4,1}(\tt,\ss)p_1(t,s)-p_{1}(\tt,\ss)q_{4,1}(t,s)} \cdot
\dfrac{q_{4,1}(t,s)}{q_{4,3}(t,s)}     \cdot\dfrac{q_{4,1}(\tt,\ss)}{q_{4,3}(\tt,\ss)}
\]
 Note that, since
 $t,s,\tt,\ss$ are independent variables, and since we have excluded planes parallel to the coordinate planes (see general assumptions in Section \ref{sec-notation}), the above rational functions are well-defined. Moreover, for every particular value $(t_0,s_0)\in \Lambda_1$  of $(\tt,\ss)$ the specialization of the rational functions are also well-defined.
Similarly, we take the polynomials
$$\ggg^{\cPw,\cA}_{i}=\numer(\chi_{i}^{\cA}(t,s,\tt,\ss)-\chi_{i}^{\cA}(h_1,h_2,\tt,\ss))$$
as well as  $$K(t,s,h_1,h_2,\tt,\ss)=\gcd(\ggg^{\cPw,\cA}_{1},\ggg^{\cPw,\cA}_{2}),$$ and $\ggg^{\cPw,\cA}$. That is, we perform Step 3 for the generic element $\cA$.

 \para

 If $\ggg^{\cPw,\cA}\neq 1$, we do not need to continue since, by Corollary \ref{cor-1-planos}, $\cS$ is a plane and the problem is trivial. Alternatively, one might avoid this case by trivially checking first whether $\cS$ is a plane.
 Therefore, we assume that $\ggg^{\cPw,\cA}=1$. Thus, we are already in Step 4 of Method 5. Nevertheless, at this stage, we know that generically $\ggg^{\cPw,\cA}=1$ but for certain $(\tt,\ss)$-values the gcd may increase the degree. Note that, because of Corollary \ref{cor-conos} and \ref{cor-conos-2}, this can only happen if $\cS$ is a cone and only for the $(\tt,\ss)$ values that generate its vertex.
  These values, if they exist, will be included in  the closed set $\Delta_{0}^{*}$. For determining $\Delta_{0}^{*}$ we consider the following direct generalized version of Lemma 3 in \cite{Sen2}:

\para

\begin{lemma}\label{lemma3-tracing} {\sf (Lemma 3 in \cite{Sen2})}   Let  $f,g\,\in\,\K[\tt,\ss][\overline{\cY}][t]\setminus \{0\}$, where $\overline{\cY}$ is a finite set of variables. Let  $f=\bar{f}\cdot \gcd(f,g)$, $g=\bar{g}
\cdot \gcd(f,g)$. Let $(t_0,s_0)\in \K^2$ be such that not both
leading coefficients of $f$ and $g$ w.r.t. $t$ vanish at $(t_0,s_0)$. If $\Resultant_{t}(\bar{f},\bar{g})$
does not vanish at $(t_0,s_0)$, then $
\gcd(f,g)(t,\overline{\cY},t_0,s_0)=\gcd((f(t,\overline{\cY},t_0,s_0),g(t,\overline{\cY},t_0,s_0)).$
\end{lemma}

\para
In our case, we consider $\ggg^{\cPw,\cA}_{i}\in \K[\tt,\ss][h_1,h_2,s][t]$; note that  both polynomials are not identically zero because we already know that $\cSw$ is not a plane.
Let $\Upsilon_i$ be the leading coefficient of $\ggg^{\cPw,\cA}_{i}$ w.r.t. $t$, and let $${\mathfrak r}(\tt,\ss,h_1,h_2,s):=\res_t(\ggg^{\cPw,\cA}_{1},\ggg^{\cPw,\cA}_{2}).$$
In addition, since we have assumed that $\ggg^{\cPw,\cA}=1$, we know that $K\in \K[h_1,h_2,\tt,\ss]$. Let $\cZ_K$
be the zero set of at least one non-zero coefficient, w.r.t. $\{h_1,h_2\}$, of the homogeneous form of maximum degree  of $K$.

We define $\Delta_{0}^{*}$ as the zero set of all  coefficients of $\Upsilon_1,\Upsilon_2$  w.r.t. $\{h_1,h_2,s\}$ union the zero set of all coefficients of ${\mathfrak r}$    w.r.t. $\{h_1,h_2,s\}$ union $\cZ_K$.

\para

Now, we proceed with Step 4 of Method 5. We have assumed that $n_1$ and $n_2$ (in Step 4 of Method 5) have been already computed. So, it only remains to analyze the determination of $m_1:=\deg(\Phi_{2,3}(\cA))$. Therefore, we apply Method 2 to $\Phi_{2,3}(\cA)(t,s)$.

\para

We assume that none of the projective curves defined by the non-constant polynomials in $\{p_1,p_2,p_3,q_{4,1},q_{4,2},q_{4,3}\}$ passes through $(0:1:0)$. If this is not the case, we perform a suitable polynomial linear change in the parameters $\{t,s\}$. Note that, in this situation,
$\Phi_{2,3}(\cA)(t,s)$ satisfies the conditions in Step 0 of Method 2, seeing the projective curves in ${\mathbb P}^2(\overline{\K(\tt,\ss)})$ where $\overline{\K(\tt,\ss)}$ is the algebraic closure of ${\K(\tt,\ss)}$.  However,  it might happen for some particular values of $\{\tt, \ss\}$  then condition fails. In order to control this, we introduce the following set $\cZ_{\infty}$. We take the homogenization (in the variables $\{t, s\}$) of the numerators and denominators of $\chi_{i}^{\cA}(t,s,\tt,\ss)$, and we substitute them in $(0:1:0)$. Observe that, as remarked above, the resulting polynomials are not identically zero. Now, $\cZ_{\infty}$ is the  union of the  zero sets  in $\K^2$ of these polynomials.

\para

In Step 1 of Method 2, we take $\chi_{i}^{\cPw,\cA}$, $i=1,2$, and in Step 2 of Method 2, we compute
\[ H_i(t,s,\cX_i,\tt,\ss)=\numer(\chi_{i}^{\cPw,\cA}-\cX_i)\in \K[\tt,\ss,\cX_1,\cX_2,t,][s]. \]
For $i=1,2$, let $M_i(t,\cX_1,\cX_2,\tt,\ss)$ be the leading coefficient of $H_i$ w.r.t. $s$. Then, we define $\Delta_{1}^{*}$ as the zero set of all
 coefficients of $M_1$   w.r.t. $\{t,\cX_1,\cX_2\}$ union the zero set of all
 coefficients of $M_2$   w.r.t. $\{t,\cX_1,\cX_2\}$ union $\cZ_{\infty}$ .

\para

\noindent In Step 3 of Method 2,  the resultant polynomial $R$  is computed. We observe that since $\ggg^{\cPw,\cA}=1$,  $R$ is not identically zero.
We see $R$
as a polynomial in $\K[\tt,\ss][t,\cX_1,\cX_2]$, and hence  we denote it  by $R(t,\cX_1,\cX_2,\tt,\ss)$.
Let $W(\cX_1,\cX_2,\tt,\ss)$ be the leading coefficient of $R$ w.r.t. $t$. Then, we define $\Delta_{2}^{*}$ as the zero set of all coefficients of $W$ w.r.t. $\{\cX_1,\cX_2\}$.

\para

\noindent In Step 4 of Method 2, first
 we express $R$ as a polynomial in $\{\cX_1,\cX_2\}$ as
\[ R=\sum_{(i,j)\in J} {\overline{a}_{i,j}}(t,\tt,\ss) \cX_{1}^{i} \cX_{2}^{j}, \]
where we collect the non-zero coefficients of $R$ w.r.t. $\{\cX_1,\cX_2\}$.
We want to control the behavior of the primitive part under specializations, which essentially means to control the content. More precisely, let
$$a(t,\tt,\ss)=\gcd(\{{\overline{a}_{i,j}}\,|\,(i,j)\in J\})=\content_{\{\cX_1,\cX_2\}}(R),$$
and let
$$ {a}_{i,j}(t,\tt,\ss)=\dfrac{{\overline{a}_{i,j}}(t,\tt,\ss)}{a(t,\tt,\ss)}.$$
Let $N(\tt,\ss)$ be the leading coefficient of $a$ w.r.t. $t$. We analyze (under specializations) the gcd of $\{ \overline{a}_{i,j}\,|\,(i,j)\in J\}$.
We distinguish several cases depending on the cardinality of $J$; we observe that  $\card(J)\not=1$ since $\deg(\Phi_{2,3}(\cPw(t_0,s_0)))>0$.



\para

\noindent {\sf [Case 1]} Let $\card(J)=2$; say $J=\{(i_0,j_0),(i_1,j_1)\}$. We apply Lemma \ref{lemma3-tracing} (i.e., the adaptation of Lemma 3 in \cite{Sen2}) to $\overline{a}_{i_0,j_0},\overline{a}_{i_1,j_1}$, seen as polynomials in $\K[\tt,\ss][t]$. Let $L_0(\tt,\ss)$ be the leading coefficient of $\overline{a}_{i_0,j_0}$ w.r.t. $t$, $L_1(\tt,\ss)$ be the leading coefficient of $\overline{a}_{i_1,j_1}$ w.r.t. $t$,  and let $S(\tt,\ss)=\res_{t}(a_{i_0,j_0},a_{i_1,j_1})$. Then, we define
$\Delta_{3}^{*}$ as the zero set of $\{L_0,L_1\}$ union the zero set of $S$, and $\Delta_{4}^{*}$ as the zero set of $N$ (see above).

\para

\noindent {\sf [Case 2]} Let $\card(J)>2$; say $J=\{(i_k,j_k)\}_{k=1,\ldots,\ell},$ with  $\ell>2$. We apply Lemma 9  in \cite{PS-grado}. For convenience of the reader we recall here the part of that lemma that we will use.

\para

\begin{lemma}\label{lemma-generalizada} {\sf (Lemma 9 in \cite{PS-grado})}  Let $f_i\,\in\,\K[\tt,\ss][t]\setminus \{0\}$, $f_i=\bar{f}_i\cdot \gcd(f_1,\ldots,f_m)$,
\,$i=1,\ldots,m$. Let $(t_0,s_0)\in \K^2$ be such that the leading
coefficient of $f_1$ w.r.t. $t$ does not vanish at $(t_0,s_0)$. If   $\Resultant_{t}(\bar{f}_1,
\bar{f}_2+\sum_{i=3}^{m}W_{i-2}\bar{f}_i)(t_0,s_0)\not=0,$ where
$W_{j},\,\,\,j=1,\ldots,m-2,$ are new variables, then
$\gcd(f_1,\ldots,f_m)(t_0,s_0)=\gcd(f_1(t_0,s_0,t),\ldots,f_m(t_0,s_0,t)).$
\end{lemma}

\para

Thus, we apply the lemma to $\{\overline{a}_{i_k,j_k}\}_{k=1,\ldots,\ell}$ seen as polynomials in $\K[\tt,\ss][t]$. Let $L(\tt,\ss)$ be the leading coefficient of $\overline{a}_{i_1,j_1}$ w.r.t. $t$,   and let $$\overline{S}(\tt,\ss,W_1,\ldots,W_{\ell-2})=\res_{t}\left(a_{i_1,j_1},a_{i_2,j_2}+\sum_{k=3}^{\ell}W_{k-2}a_{i_k,j_k}\right).$$
We define $\Delta_{3}^{*}$ as the zero set of all  coefficients of $\overline{S}$ w.r.t. $\{W_1,\ldots,W_{\ell-2}\}$,
and $\Delta_{4}^{*}$ as the zero set of $L$ union the zero set of $N$.

\para

Note that, since $\cS$ is irreducible, and $\cPw(t,s)$ is a generic element of $\cSw$, we have the following lemma.

\para

\begin{lemma}\label{lemma-final}
Let $(t_0,s_0)\in \Lambda_1$ be such that
\[ \deg(\Phi_{2,3}(\cPw(t_0,s_0)))=\deg_t\left(\frac{R(t,\cX_1,\cX_2,\tt,\ss)}{a(t,\tt,\ss)}\right). \]It holds that $\cPw(t_0,s_0)$ is a simple point of $\cSw$.
\end{lemma}

\para

We finish this subsection with the following theorem.

\para

\begin{theorem}\label{th-omega}
 $\forall \, (t_0,s_0)\in \Lambda_1\setminus \Delta^{*}$, $\cPw(t_0,s_0)$ is a simple point of $\cSw$.
\end{theorem}
\begin{proof} Let $(t_0,s_0)\in \Lambda_1\setminus \Delta^{*}$; throughout the proof, we denote $\cPw(t_0,s_0)$ by $\cA_0$. Since $(t_0,s_0)\in \Lambda_1$, then $(t_0,s_0)\not\in \Delta_4$, and hence  $\cA_0$ is well defined and it is  a point on $\cSw$. Moreover, $\chi_{i}^{\cA_0}(t,s)$ are also well-defined. On the other hand, since $(t_0,s_0)\not\in \Delta_{0}^{*}$, then ${\mathfrak r}(t_0,s_0,h_1,h_2,s)$ does not vanish and at least one the polynomials $\Upsilon_1(t_0,s_0,h_1,h_2,s)$, $\Upsilon_2(t_0,s_0,h_1,h_2,s)$, does not vanish. Then, by Lemma \ref{lemma3-tracing},
\[ K(t,s,h_1,h_2,t_0,s_0)=\gcd(\ggg_{1}^{\cPw,\cA_0},\ggg_{2}^{\cPw,\cA_0}). \]
Furthermore, since $(t_0,s_0)\not\in\Delta_{0}^{*}$, then $(t_0,s_0)\not\in \cZ_K$, and hence
\[ \deg_{\{t,s\}}(K(t,s,h_1,h_2,\tt,\ss))=\deg_{\{t,s\}}(K(t,s,h_1,h_2,t_0,s_0)). \]
Therefore,
\[ \ggg^{\cPw,\cA_0}(t,s,h_1,h_2)=\ggg^{\cPw,\cA}(t,s,h_1,h_2,s_0,t_0)=1.\]
Note that since $(t_0,s_0)\not\in \Delta_{1}^{*}$, then $(t_0,s_0)\not\in \cZ_{\infty}$ and hence the conditions in Step 0, Method 2, are satisfied. Moreover,  neither $M_1(t,\cX_1,\cX_2,t_0,s_0)$ nor $M_2(t,\cX_1,\cX_2,t_0,s_0)$ vanish. Similarly, since
$(t_0,s_0)\not\in \Delta_{2}^{*}$, $W(\cX_1,\cX_2,t_0,s_0)$ does not vanish.

\noindent If we are in {\sf case 1}, since $(t_0,s_0)\not\in \Delta_{3}^{*}$ we get that  $L_0(t_0,s_0)\neq 0$ or $L_1(t_0,s_0)\neq 0$, and $S(t_0,s_0)\neq 0$. Thus, by Lemma \ref{lemma3-tracing}, we get that
$a(t,t_0,s_0)=\gcd({\overline{a}_{i_0,j_0}}(t,t_0,s_0),{\overline{a}_{i_1,j_1}}(t,t_0,s_0))$. Moreover, by well-know properties of resultants, we get that (up to multiplication by a non-zero constant), $R(t,\cX_1,\cX_2,t_0,s_0)=\res_s(H_1(t,s,\cX_1,t_0,s_0),H_2(t,s,\cX_2,t_0,s_0))$.  Furthermore, since $W(\cX_1,\cX_2,t_0,s_0)\neq 0$ (see above),
\[ \deg_t(R(t,\cX_1,\cX_2,t_0,s_0))=\deg_t(R(t,\cX_1,\cX_2,\tt,\ss)).\]
On the other hand, $(t_0,s_0)\not\in \Delta_{4}^{*}$ implies that $\deg_t(a(t,\tt,\ss))=\deg_t(a(t,t_0,s_0))$. Summarizing,
\[\deg(\Phi_{2,3}(\cA_0))= \deg_t(\pp_{\{\cX_1,\cX_2\}}(R(t,\cX_1,\cX_2,t_0,s_0)))=\]
\[\deg_t(\pp_{\{\cX_1,\cX_2\}}(R(t,\cX_1,\cX_2,\tt,\ss)))=\deg(\Phi_{2,3}(\cA)). \]
Therefore, by Lemma \ref{lemma-final}, $\cA_0$ is simple.

\noindent If we are in {\sf case 2},
since $(t_0,s_0)\not\in \Delta_{3}^{*}$ we get that  $\overline{S}(t_0,s_0,W_1,\ldots,W_{\ell-2})\neq 0$.  Since $(t_0,s_0)\not\in \Delta_{4}^{*}$ we know that $L(t_0,s_0)\neq 0$ and $N(t_0,s_0)\neq 0$.  Thus, by Lemma \ref{lemma-generalizada}, we get that
$a(t,t_0,s_0)=\gcd(\{{\overline{a}_{i,j}(t,t_0,s_0)\,|\,(i,j)\in J\}})$.
  From here the proof follows as in the {\sf case 1}.
 \end{proof}

\para

 \begin{center}
 {\sf First level (Part II): decomposition of $\Delta^*\setminus \Delta_4$}
 \end{center}

\para

We  decompose $\Delta^*$ as union of irreducible closed sets; note that they are of dimension less or equal 1.
Let $\cc$ be an irreducible curve in $\Delta^*$. If $\cc\subset \Delta_4$, there is nothing to do. If not,
we compute the intersection  of $\Delta_4$  and $\cc$ (note that $\Delta_4$ is empty or a plane curve). This intersection would be either empty or finitely many points. For an open subset  of $\cc$, the degree of the corresponding map $\Phi_{2,3}$ would be invariant, and hence all points in the open subset would generate points on $\cS$  with the same multiplicity. The complementary of this open subset is now either empty or a finite set of points. So, if it is not empty, we apply the formula to each of the finitely many points in the closed set as well as for those points in the zero-dimensional components of $\Delta^*$.

\para

In order to compute the open subset of $\cc$, we do an analogous reasoning as in the previous subsection.

\para

\noindent {\sf [Rational case]} If $\cc$ is rational, we compute a proper normal rational parametrization $\cQ(\lambda)$  of $\cc$  (see \cite{libro}). Then, we apply Method 5 to $\cH(\lambda):=\cPw(\cQ(\lambda))$; say that $\cH(\lambda)$ is expressed as:
$$\cH(\lambda)=\left(\dfrac{\varphi_1(\lambda)}{\phi_1(\lambda)},\dfrac{\varphi_2(\lambda)}{\phi_2(\lambda)},
\dfrac{\varphi_3(\lambda)}{\phi_3(\lambda)}
\right)$$
where $\gcd(\varphi_i,\phi_i)=1, i=1,2,3$. Note that Step 1 as well $n_1, n_2$ (in Steps 3, 4) were already computed in Level I (first part). In Step 2, we have to compute $\ggg^{\cPw,\cH(\lambda)}$. For that we distinguish two cases:
\begin{enumerate}
\item  if $\cc\not\subset \Delta^{*}_{0}$ (see first part of the proof of Theorem \ref{th-omega}) then, if $(t_0,s_0):=\cQ(\lambda_0)\not\in {(\cc\cap \Delta^{*}_{0})\setminus \Delta_4}$, it holds that $\ggg^{\cPw,\cH(\lambda_0)}=1$. For the others, the finitely many (maybe empty) points in ${(\cc\cap \Delta^{*}_{0})\setminus \Delta_4}$, one applies directly the whole Method 5.  Observe the connection with cones; see Corollaries \ref{cor-conos} and \ref{cor-conos-2}.
    \item If $\cc\subset \Delta^{*}_{0}$, we repeat the reasoning done for the computation of $\Delta^*$ in Part I of Level 1st.
    That is, we compute $\ggg^{\cPw,\cH(\lambda)}$ generically (i.e.,   treating $\lambda$ as a transcendental element).
    Note that this, essentially, means computing a gcd in unique factorization domain $\K[\lambda,h_1,h_2,s,t]$ that
    (see e.g. Section 4.1. in \cite{Winkler}) can be reduced to the computation in the Euclidean domain $\K(\lambda,h_1,h_2,s)[t]$.
    For  an open subset  $\tilde{\Delta}$ of $\cc$, $\ggg^{\cPw,\cH(\lambda)}=1$, and we can go ahead through Step 4.
    For the complementary closed set (that is empty or finite) we execute the whole Method 5.
\end{enumerate}
Therefore, after performing the above considerations, we can assume that $\ggg^{\cPw,\cH(\lambda)}=1$. Thus, we pass to Step 4, and hence it only remains to apply Method 2 to compute $\deg(\Phi_{2,3}(\cH(\lambda))$, where $\lambda$ belongs to a non-empty open subset of $\K$; namely those $\lambda$ such that ${\cQ}(\lambda)\in \cc \setminus \Delta^{*}_{0}$ if we come from case 1 (above) or ${\cQ}(\lambda)\in\cc\setminus \tilde{\Delta}$ if we come from case 2 (above). We observe that all computations can be carried out:  we have to compute resultants in the unique factorization domain $\K[\lambda,t,\cX_1,\cX_2][s]$ and gcds in the Euclidean domain $\K(\lambda)[t]$.

\para

\noindent {\sf [Positive genus case]} If $\cc$ is not rational, we  work over the field of rational functions $\K(\cc)$ of the curve (see \cite{libro}). Let $f(t,s)$ be the defining polynomial of $\cc$, then $\K(\cc)$ is the quotient field of $\K[t,s]/(f)$.
Then, we apply Method 5 to $\cP(\tv,\sv)$, where  $\tv,\sv\in \K(\cc)$ are representatives of the equivalent classes of $t,s$ respectively, i.e., $\tv-t$, and $\sv-s$ belong to the ideal $(f)$. We recall  that the arithmetic in the field $\K(\cc)$ can be executed by using
the defining polynomial of $\cc$. We observe that all computations can be carried out:  we have to compute gcds in $\K(\cc)[h_1,h_2,s,t]$ (which can be performed in the Euclidean domain $\K(\cc)(h_1,h_2)[t]$),
 resultants in the unique factorization domain $\K(\cc)[t,\cX_1,\cX_2][s]$ and gcds in the Euclidean domain $\K(\cc)[t]$.

\para

For each 1-dimensional component $\cc$ of $\Delta^*$ we will get an open subset where all points (i.e., parameter values) behave the same; that is all have the same multiplicity. So each of these open subsets will be part of $\cF_{k}^{1}$ for some $k$. The complementary of these open sets are either empty or zero-dimensional. So we will have, in the worst case, a set of finitely many parameter values to be classified. For each of them we apply Method 5, and we determine their multiplicity. Finally, they are included in the corresponding $\cF_{k}^{1}$.

\para

\begin{center}
{\sf Second, third and fourth levels}
\end{center}

\para

Let  $\Delta_4\setminus \cB_a\neq \emptyset$. We want to decompose $\Lambda_2$ (i.e., $\Delta_4\setminus \Delta_3$). We observe that $\Lambda_2$ would be either empty or $1$-dimensional; since $\Delta_i$ are either empty or plane curves. Clearly, the interesting case is when $\dim(\Lambda_2)=1$. Then, for each irreducible component of $\Lambda_2$ we proceed as in the first level (part II). Finally, note that the same argument and strategy is valid for the third and the fourth levels.

\section{Computing the  $\cP$-singularities at infinity}\label{sec-singular-infinity}

In this section, we show how to proceed with the steps 2 and 3 of our strategy (see Section \ref{sec-notation}). So, first we analyze whether $A=(0:1:0)$ is a $\cP$-singularity. For this purpose, we  check whether $A\in \cB$. If $A\not\in \cB$, then at least one of the polynomials $\cpp_i$ does not vanish on $A$ (say  w.l.o.g. $\cpp_4$). Then, we  replace $\cP(\ttt,\sss,\vvv)$ by
\[\left(\frac{\cpp_1(\frac{\ttt}{\sss},1,\frac{\vvv}{\sss})}{\cpp_4(\frac{\ttt}{\sss},1,\frac{\vvv}{\sss})}:
\frac{\cpp_2(\frac{\ttt}{\sss},1,\frac{\vvv}{\sss})}{\cpp_4(\frac{\ttt}{\sss},1,\frac{\vvv}{\sss})}:\frac{\cpp_3(\frac{\ttt}{\sss},1,\frac{\vvv}{\sss})}{ \cpp_4(\frac{\ttt}{\sss},1,\frac{\vvv}{\sss})}:1\right) \]
So, introducing the notation $\{\tilde{t}=\frac{\ttt}{\sss},\tilde{v}=\frac{\vvv}{\sss}\}$ and $\tilde{p}_i(\tilde{t},\tilde{v})=\cpp_i(\tilde{t},1,\tilde{v})$, we get
\[ \tilde{\cP}_{x_4}(\tilde{t},\tilde{v})=\left(\dfrac{\tilde{{p}}_1(\tilde{t},\tilde{v})}{\tilde{{p}}_4(\tilde{t},\tilde{v})},
\dfrac{\tilde{{p}}_2(\tilde{t},\tilde{v})}{\tilde{{p}}_4(\tilde{t},\tilde{v})},
\dfrac{\tilde{{p}}_3(\tilde{t},\tilde{v})}{\tilde{{p}}_4(\tilde{t},\tilde{v})}  \right)\]
that parametrizes $\cSw$.  Similarly, if necessary, we introduce $\tilde{\cP}_{x_i}(\tilde{t},\tilde{v})$ with $i=1,2,3$. Now, we  apply Method 5 to compute $${\mult((0:1:0))=\mult(\tilde{\cP}_{x_4}(0,0),\cSw)=\mult(\cP(0,1,0),\cS).}$$

\para

Now, it only remains to analyze the points in
${\cal E}=\{(1:\lambda_0:0)\,|\, \lambda_0\in \K\}.$
For that, first we determine those points in $\cal E$ that are base points, namely
\[ {\cal E}^*=\{(1:\lambda_0:0)\, |\,\gcd(\cpp_1(1,\lambda,0),\cpp_2(1,\lambda,0),\cpp_3(1,\lambda,0),\cpp_4(1,\lambda,0))(\lambda_0)=0\} \]
There exists  $i$ such that $\cpp_i(1,\lambda,0)$ is not identically zero, since otherwise $\vvv$ would divide  $\gcd(\cpp_1,\ldots,\cpp_4)$, which is a contradiction. Let us assume w.l.o.g. that $\cpp_4(1,\lambda,0)$ is not identically zero.
 We then introduce the finite set
\[ {\cal E}^{**}=\{(1:\lambda_0:0)\, |\,\cpp_4(1,\lambda_0,0)=0\}\setminus {\cal E}^*, \]
and we proceed to compute the multiplicity of each $(1:\lambda_0:0)\in {\cal E}^{**}$. For that, we observe that there exists $j\neq 4$ such that $\cpp_j(1,\lambda_0,0)\neq 0$, and we  apply the multiplicity formula using the dehomogenization of $\cP(\ttt,\sss,\vvv)$ w.r.t. the $j$-component.

\para

To analyze the open subset ${\cal E}\setminus  {\cal E}^{**}$, we replace $\cP(\ttt,\sss,\vvv)$ by
\[\left(\frac{\cpp_1(1,\frac{\sss}{\ttt},\frac{\vvv}{\ttt})}{\cpp_4(1,\frac{\sss}{\ttt},\frac{\vvv}{\ttt})}:
\frac{\cpp_2(1,\frac{\sss}{\ttt},\frac{\vvv}{\ttt})}{\cpp_4(1,\frac{\sss}{\ttt},\frac{\vvv}{\ttt})}:\frac{\cpp_3(1,\frac{\sss}{\ttt},\frac{\vvv}{\ttt})}{ \cpp_4(1,\frac{\sss}{\ttt},\frac{\vvv}{\ttt})}:1\right) \]
So, introducing the notation $\{\hat{s}=\frac{\sss}{\ttt},\hat{v}=\frac{\vvv}{\ttt}\}$ and $\hat{p}_i(\hat{s},\hat{v})=\cpp_i(1,\hat{s},\hat{v})$, we get
\[ \hat{\cP}_{x_4}(\hat{s},\hat{v})=\left(\dfrac{\hat{p}_1(\hat{s},\hat{v})}{\hat{p}_4(\hat{s},\hat{v})},
\dfrac{\hat{p}_2(\hat{s},\hat{v})}{\hat{p}_4(\hat{s},\hat{v})},
\dfrac{\hat{p}_3(\hat{s},\hat{v})}{\hat{p}_4(\hat{s},\hat{v})}  \right)\]
that parametrizes $\cSw$. Similarly, if necessary, we introduce $\hat{\cP}_{x_i}(\hat{s},\hat{v})$ with $i=1,2,3$. Now, one has to proceed as in Section \ref{sec-sing}, Level 1 (Part II, case rational) with the rational curve ${\cal Q}(\lambda)=(\lambda,0)$.

\section{Algorithm and Example}\label{sec-algorithm}

 In this section we summarize all the previous ideas to derive an algorithm that we illustrate with a complete example. For this purpose, let  $\cS\subset \projtres$ be a projective surface, and $\cP(\ttt,\sss,\vvv)$ a parametrization of $\cS$  expressed as
\[ \cP(\ttt,\sss,\vvv)=\left(\cpp_1(\ttt,\sss,\vvv):\cpp_2(\ttt,\sss,\vvv):\cpp_3(\ttt,\sss,\vvv):\cpp_4(\ttt,\sss,\vvv)\right)\]
where $\cpp_i\in \K[\ttt,\sss,\vvv]$ are homogeneous polynomials of the same degree, and
 $\gcd(\cpp_1,\cpp_2,\cpp_3,\cpp_4)=1$. Let $\cB$ the zero set in $\projdos$ of $\{\cpp_1,\ldots,\cpp_4\}$. Then, the
 algorithm  decomposes $\projdos\setminus \cB$ as $$\projdos\setminus \cB=\cup_{k=1}^{\ell} \cSm_k$$ such that if $(\ttt_0:\sss_0:\vvv_0)\in \cSm_k$ then $\cP(\ttt_0,\sss_0,\vvv_0)$ is a point of $\cS$ of multiplicity $k$.

\para

As already remarked in Section \ref{sec-notation}, we assume that none of the polynomials $\cpp_i$ is zero or, more generally, that there do not exist
 $\cpp_i,\cpp_j$ and  $\lambda\in \K$ such that $\cpp_i=\lambda \cpp_j$. Note that this excluded situation corresponds to a plane,
  and hence $\cSm_1=\projdos$.

  \para

  In addition, we use the notation introduced in Section \ref{sec-notation}, namely the affine surfaces $\cS_{x_i}$, the affine rational parametrizations  $\cP_{x_i}(t,s)$, and  the polynomials $p_k, q_{i,j}$. Moreover, we also use the notation $\tilde{\cP}_{x_i}(\tilde{t},\tilde{v})$, $\hat{\cP}_{x_i}(\hat{s},\hat{v})$ (see Section \ref{sec-singular-infinity}). In this situation, the algorithm is as follows.

\para

\noindent \underline{\sf Algorithm}

\para

\hspace*{2cm} {\sf [Preparatory Steps]}

\vspace{1 mm}

\noindent {\sf [Step 0.]} If any of the projective curves defined by the non-constant polynomials in $\{p_1,p_2,p_3,q_{4,1},q_{4,2},q_{4,3}\}$ of the parametrization $\cPw$ passes through $(0:1:0)$ we perform a suitable polynomial linear change in the parameters $\{t,s\}$.

\vspace{1 mm}

\noindent {\sf [Step 1.]}  Apply Method 1 to compute $n_2:=\deg(\cPw)$ (see Section \ref{sec-formula}).

\vspace{1 mm}

\noindent {\sf [Step 2.]} Apply Method 4 to determine an affine point, say $A_0$, out of the affine surface $\cSw$  (take $A_0$ with non-zero components such that if the algorithm, in subsequent steps, requires a point in $\K^3\setminus \cS_{x_i}$ with $i\neq 4$ no further computation would be needed (see Remark \ref{rem-A0}) and apply Method 2 to compute $n_1:=\Phi_{2,3}(A_0)$ (see Section \ref{sec-formula}).

\vspace{1 mm}

\noindent {\sf [Step 3.]}
 Let $\Delta_i:=\{(t_0,s_0)\,|\,\lcm(q_{i,1},q_{i,2},q_{i,3})(t_0,s_0)=0\}, \,i=1,\ldots,4$.

\vspace{1 mm}

\noindent {\sf [Step 4.]}
 Let $\cB_a=\bigcap_{i=1}^{4} \Delta_i$ (see Lemma \ref{lemma-ptos-base-afines}) find $\cB$.

\para

\hspace*{2cm} {\sf [$\cP$-affine singularities (First level: part I)]}

\vspace{1 mm}

\noindent {\sf [Step 5.]} Compute (see Section \ref{sec-sing}) $\cA=\cPw(\tt,\ss)$; $\Phi_{2,3}(\cA)=(\chi_{1}^{\cA},\chi_{2}^{\cA})$;
$\ggg^{\cPw,\cA}_{i}=\numer(\chi_{i}^{\cA}(t,s,\tt,\ss)-\chi_{i}^{\cA}(h_1,h_2,\tt,\ss))$;
as well as  $K=\gcd(\ggg^{\cPw,\cA}_{1},\ggg^{\cPw,\cA}_{2}).$

\vspace{1 mm}

\noindent {\sf [Step 6.]} If $\deg_{\{t,s\}}(K)>0$ then {\sf return} $\cSm_1=\projdos$ ($\cS$ is a  plane).

\vspace{1 mm}

\noindent  {\sf [Step 7.]} Computation of  $\Delta_{0}^{*}$

\noindent\hspace*{2 mm} \parbox{15 cm} {{\sf [Step 7.1.]}
Compute the leading coefficient  $\Upsilon_i$ of $\ggg^{\cPw,\cA}_{i}$ w.r.t. $t$,

\noindent {\sf [Step 7.2.]} Compute ${\mathfrak r}:=\res_t(\ggg^{\cPw,\cA}_{1},\ggg^{\cPw,\cA}_{2}).$

\noindent {\sf [Step 7.3.]} $\cZ_K$
 is the zero set of at least one non-zero coefficient, w.r.t. $\{h_1,h_2\}$, of the homogeneous form of maximum degree  of $K$.

\noindent {\sf [Step 7.4.]} $\Delta_{0}^{*}$ is the zero set of all  coefficients of $\Upsilon_1,\Upsilon_2$  w.r.t. $\{h_1,h_2,s\}$ union the zero set of all coefficients of ${\mathfrak r}$    w.r.t. $\{h_1,h_2,s\}$ union $\cZ_K$. }

\noindent  {\sf [Step 8.]} Computation of  $\Delta_{1}^{*}$

\noindent\hspace*{2 mm} \parbox{15 cm} {{\sf [Step 8.1.]} Homogenize (w.r.t. $\{t, s\}$)  the numerators and denominators of $\chi_{i}^{\cA}$, and  substitute them in $(0:1:0)$. Take $\cZ_{\infty}$ as the union of the zero sets in $\K^2$ of these polynomials.

\noindent {\sf [Step 8.2.]}  Compute $H_i=\numer(\chi_{i}^{\cPw,\cA}-\cX_i)$ as well as  the leading coefficient $M_i$ of $H_i$ w.r.t. $s$.

\noindent {\sf [Step 8.3.]}  $\Delta_{1}^{*}$ is the zero set of all
 coefficients of $M_1$   w.r.t. $\{t,\cX_1,\cX_2\}$ union the zero set of all
 coefficients of $M_2$   w.r.t. $\{t,\cX_1,\cX_2\}$ union $\cZ_{\infty}$.
}

\noindent  {\sf [Step 9.]} Computation of  $\Delta_{2}^{*}$

\noindent\hspace*{2 mm} \parbox{15 cm}{{\sf [Step 9.1.]} Compute  $R:=\res_s(H_1,H_2)$ and  its leading coefficient $W$  w.r.t. $t$.

\noindent {\sf [Step 9.2.]}  $\Delta_{2}^{*}$ is the zero set of all coefficients of $W$ w.r.t. $\{\cX_1,\cX_2\}$.
}

\noindent  {\sf [Step 10.]} Computation of  $\Delta_{3}^{*}$ and $\Delta_{4}^{*}$

\noindent\hspace*{2 mm} \parbox{15 cm}{{\sf [Step 10.1.]} Compute the set $\{\overline{a}_{i,j}\,|\,(i,j)\in J\}$ of all  coefficients of $R$ w.r.t. $\{\cX_1,\cX_2\}$.

\noindent {\sf [Step 10.2.]} Compute $a=\gcd(\{\overline{a}_{i,j}\,|\,(i,j)\in J\})$ and $a_{i,j}=\frac{\overline{a}_{i,j}}{a}$.

\noindent {\sf [Step 10.3.]} Determine the leading coefficient $N$ of $a$ w.r.t. $t$.

}

\noindent\hspace*{2 mm} \parbox{15 cm}{  {\sf [Step 10.4.]} If $\card(J)=2$ (say $J=\{(i_0,j_0),(i_1,j_1)\}$)

\noindent\hspace*{4 mm} \parbox{15 cm}{{\sf [Step 10.4.1.]} Compute  the leading coefficient $L_j$ of $a_{i_j,j_j}$ w.r.t. $t$ $(j=0,1)$  and  $S=\res_{t}(a_{i_0,j_0},a_{i_1,j_1})$.

\noindent {\sf [Step 10.4.2.]} $\Delta_{3}^{*}$ is the zero set of $\{L_0,L_1\}$ union the zero set of $S$.

\noindent {\sf [Step 10.4.3.]}  $\Delta_{4}^{*}$ is the zero set of $N$.}

\noindent {\sf [Step 10.5.]} If $\card(J)>2$ (say  $J=\{(i_k,j_k)\}_{k=1,\ldots,\ell}$)

\noindent\hspace*{4 mm} \parbox{15 cm}{{\sf [Step 10.5.1.]} Compute  the leading coefficient
 $L$  of $a_{i_1,j_1}$ w.r.t. $t$ and $\overline{S}=\res_{t}(a_{i_1,j_1},a_{i_2,j_2}+\sum_{k=3}^{\ell}W_{k-2}a_{i_k,j_k}).$

\noindent {\sf [Step 10.5.2.]} $\Delta_{3}^{*}$ is the zero set of all  coefficients of $\overline{S}$ w.r.t. $\{W_1,\ldots,W_{\ell-2}\}$.

\noindent {\sf [Step 10.5.3.]}  $\Delta_{4}^{*}$ is the zero set of $L$ union the zero set of $N$.}
}

\noindent {\sf [Step 11.]} Set $\Delta^*=\cup_{i=0}^{4} \Delta_{i}^{*}$, and include $\mathbf{j}((\K^2\setminus\Delta_4) \setminus \Delta^*)$ in $\cSm_1$.
\para

\hspace*{2cm} {\sf [$\cP$-affine singularities (First level: part II)]}

\vspace{1 mm}

\noindent {\sf [Step 12.]} Decompose $\Delta^*$ into irreducible components.

\noindent {\sf [Step 13.]} For each point $A$ at a zero-dimensional component of $\Delta^*$, if $A\not\in \Delta_4$ then apply Method 5 to compute $\alpha=\mult(A,\cSw)$, and include $\mathbf{j}(A)$ in
$\cSm_{\alpha}$.

\noindent {\sf [Step 14.]} For each 1-dimensional irreducible component $\cc$ of $\Delta^*$, compute its genus.

\noindent {\sf [Step 15.]} If $\cc$ is rational proceed as in Section \ref{sec-sing} (Level 1, Part II, rational case). This will generate an open subset $\cc^*$ of $\cc$ where the multiplicity is invariant and that would be included, via $\mathbf j$, in the corresponding $\cSm_k$. For the finitely many points in the $\cc \setminus \cc^*$, proceed as in Step 13.

\noindent {\sf [Step 16.]} If $\cc$ is not rational proceed as in Section \ref{sec-sing} (Level 1, Part II, positive genus case). This will generate an open subset $\cc^*$ of $\cc$ where the multiplicity is invariant and that would be included, via $\mathbf j$, in the corresponding $\cSm_k$. For the finitely many points in the $\cc \setminus \cc^*$, proceed as in Step 13.

\para

\hspace*{2cm} {\sf [$\cP$-affine singularities (Second, Third and Fourth Level)]}

\vspace{1 mm}

\noindent {\sf [Step 17.]}  If $\Delta_4 \setminus \cB_a =\emptyset$ (see Step 4)
 go to Step 20 else  proceed as follows

\noindent\hspace*{2 mm} \parbox{15 cm}{{\sf [Step 17.1.]} If  $\Delta_4\setminus \Delta_3=\emptyset$ go to Step 18.

\noindent {\sf [Step 17.2.]} Compute the irreducible decomposition of $\Delta_4\setminus \Delta_3$.

\noindent {\sf [Step 17.3.]} Proceed as in Steps 13, 14, 15, using $\cPz$ instead of $\cPw$. }

\noindent {\sf [Step 18.]}  If $(\Delta_4\cap \Delta_3) \setminus \cB_a= \emptyset$ go to Step 20 else  proceed as follows

\noindent\hspace*{2 mm} \parbox{15 cm}{{\sf [Step 18.1.]} If  $(\Delta_4\cap \Delta_3)\setminus \Delta_2=\emptyset$ go to Step 19.

\noindent {\sf [Step 18.2.]} Compute the irreducible decomposition of $(\Delta_4\cap \Delta_3)\setminus \Delta_2$.

\noindent {\sf [Step 18.3.]} Proceed as in Steps 13, 14, 15, using $\cPy$ instead of $\cPw$. }

\noindent {\sf [Step 19.]}  If $(\Delta_4\cap \Delta_3 \cap \Delta_2) \setminus \cB_a= \emptyset$ go to Step 20 else  proceed as follows

\noindent\hspace*{2 mm} \parbox{15 cm}{{\sf [Step 19.1.]} If  $(\Delta_4\cap \Delta_3\cap \Delta_2)\setminus \Delta_1=\emptyset$ go to Step 20.

\noindent {\sf [Step 19.2.]} Compute the irreducible decomposition of  $(\Delta_4\cap \Delta_3\cap \Delta_2)\setminus \Delta_1$.

\noindent {\sf [Step 19.3.]} Proceed as in Steps 13, 14, 15, using $\cPx$ instead of $\cPw$. }

\para

\hspace*{2cm} {\sf [$\cP$  singularities  at infinity]}

\vspace{1 mm}

\noindent {\sf [Step 20.]} If $(0:1:0)\not\in \cB$ (i.e., not all $\cpp_i(0,1,0)$ vanish) apply Method 5 to compute $\alpha:=\mult(\tilde{\cP}_{x_4}(0,0),\cSw)$  and include $(0:1:0)$ in
$\cSm_\alpha$; we are assuming that $\cpp_4(0,1,0)\neq 0$, otherwise take other component and proceed accordingly.

\noindent {\sf [Step 21.]} Check whether $\cpp_4(1,\lambda,0)$ does not vanish. If it does vanish, find  $\cpp_i$ not vanishing at $(1:\lambda:0)$ and proceed accordingly.

\noindent {\sf [Step 22.]} For each $(1:\lambda_0:0)$ such that $\cpp_4(1,\lambda_0,0)=0$ if $(1:\lambda_0:0)\not\in\cB$: find $\cpp_j$ such that $\cpp_j(1,\lambda_0,0)\neq 0$, compute  $\alpha:=\mult(\hat{\cP}_{x_j}(\lambda_0,0),\cS_{x_j})$, and include $(0:1:0)$ in
$\cSm_\alpha$.
\noindent {\sf [Step 23.]} Proceed as in Step 15 using $\hat{\cP}_{x_4}(\hat{s},\hat{v})$, instead of $\cPw$, and the curve $(\lambda, 0)$.

\para

\begin{example}\label{ex-1}
We consider the parametrization
\[ \cP(\ttt, \sss, \vvv)=\left({\sss}^{2}:{\sss}^{2}+{\ttt}^{2}+{\vvv}^{2}: \left( {
\ttt}+2\,\sss \right) \vvv: \left( \sss+\ttt
 \right) \vvv\right) \]
of the surface $\cS$.
One can easily check that the parametrization satisfies all hypotheses in Section \ref{sec-notation}. In addition
\begin{multicols}{2}
\begin{itemize}
\item[] $\cPx=\left({\dfrac {{s}^{2}+{t}^{2}+1}{{s}^{2}}},{\dfrac {t+2\,s}{{s}^{2}}},{
\dfrac{s+t}{{s}^{2}}} \right)$
\item[] $\cPy=\left( {\dfrac {{s}^{2}}{{s}^{2}+{t}^{2}+1}},{\dfrac {t+2\,s}{{s}^{2}+{t}^{2}+
1}},{\dfrac {s+t}{{s}^{2}+{t}^{2}+1}} \right)$
\item[] $\cPz=\left({\dfrac{{s}^{2}}{t+2\,s}},{\dfrac {{s}^{2}+{t}^{2}+1}{t+2\,s}},{\dfrac
{s+t}{t+2\,s}}\right)$
\item[] $\,\,\,\,\,\cPw=\left({\dfrac {{s}^{2}}{s+t}},{\dfrac {{s}^{2}+{t}^{2}+1}{s+t}},{\dfrac{t+2\,
s}{s+t}}  \right).$
\end{itemize}
\end{multicols}
\noindent Note that $\cPw$ satisfies the hypotheses in Step 0. In Step 1 one gets $n_2:=1$, and in Step 2 we get $A_0:=(1,1,1)$ and $n_1:=4$. In Step 3 we get that
\begin{itemize}
\item $\Delta_1$ is the line $s=0$,
\item $\Delta_2$ is the complex circle $s^2+t^2+1=0$,
\item $\Delta_3$ is the line $t+2s=0$, and
\item $\Delta_4$ is the line $s+t=0$.
\end{itemize}
Therefore, in Step 4 we get $\cB_a=\cB=\emptyset$. In Step 5 we get that $K=\tt+\ss$. We start the computation of $\Delta_{0}^{*}$. In Step 7.1. we get
$$ \Upsilon_1=\left( \ss+\tt \right)  \left( -{{ h_2}}^{2}\ss-{{h_2}}^{2}\tt+{\ss}^{2}{h_2}+{\ss}^{2}{h_1} \right), \Upsilon_2= \left( \ss+\tt \right) { h_2}\,\ss\,
 \left( -\ss+{ h_2} \right).
$$
In Step 7.2.

\noindent ${\mathfrak r}=-(\ss+\tt)^{3} (s-\ss)  (-s+h_{2})  ( -h_{2}^{2}\ss-h_{2}^{2}\tt+{\ss}^{2}h_{2}+{\ss}^{2
}h_{1})  ( -{h_{1}}^{2}{\ss}^{3}s+{h_{1}}
^{2}{\ss}^{2}{s}^{2}+{\ss}^{2}{h_{1}}^{2}sh_{2}-
2\,h_{2}\,\tt\,{s}^{2}\ss\,h_{1}+{\ss}^{
3}h_{2}-2\,h_{2}^{2}{\ss}^{2}+{\tt}^{2}sh_{2}^{2}\ss+h_{2}^{3}\ss+{s}^{2}{\tt}^{2}
h_{2}^{2}-s{\tt}^{2}h_{2}^{3}).
 $

\noindent In Step 7.3. $\cZ_K$ is the line $\ss+\tt=0$. Finally, in Step 7.4. we conclude that
$$\Delta_{0}^{*}=\{(\tt,\ss)\in \K^2\,|\, \tt+\ss=0.\}$$
In Step 8. we compute $\Delta_{1}^{*}$. In Step 8.1. we get that $\cZ_\infty$ is union of the lines $\tt+\ss=0$ and $\tt=0$. In Step 8.2. we get
$M_1=\ss+\tt-\cX_1(\ss+\tt),\, M_2=-\cX_2(\ss+\tt). $ So, in Step 8.3. we conclude that
\[ \Delta_{1}^{*}=\Delta_{0}^{*}\cup\{(\tt,\ss)\in \K^2\,|\, \tt=0\}.\]
For the computation of $\Delta_{2}^{*}$, in Step 9.1. we get $W=\cX_{2}^{2}(\ss+\tt)^4$. Therefore,
\[ \Delta_{2}^{*}=\Delta_{0}^{*}.\]
For computing $\Delta_{3}^{*}$ and $\Delta_{4}^{*}$, in Step 10.1., we get that $R$ has 6 non-zero coefficients w.r.t. $\{\cX_1,\cX_2\}$. Moreover,
$a=(\ss+\tt)^2(t-\tt) $ (see Step 10.2.) and $N=(\ss+\tt)^2$ (see Step 10.3.). Since $\card(J)=6$ we go through Step 10.5. Then (see Step 10.5.1.), $a_{i_1,j_1}=t(\ss^2+\tt^2)-\tt$, $L=\ss^2+\tt^2$, and

\noindent $\overline{S}= -\tt\,\ss^{6}-2\,\tt^{3}\ss^{4}-
\tt^{5}\ss^{2}+2\,W_{1}\,\ss^{3}\tt^{3}-W_{1}\,\tt^{4}\ss^{2}+2\,\ss^{2}W_{2}\,\tt^{4}+2\,\ss^{4}W_{2}\,{
\tt}^{2}-2\,\ss^{3}W_{2}\,\tt^{3}-2\,W_4\,\ss^{3}\tt^{2}+W_{4}\,\tt^{3
}\ss^{2}+W_{4}\,\ss^{4}\tt+2\,\tt^{3}W_{4}\,\ss^{4}+2\,\tt^{5}W_{4}\,
\ss^{2}+2\,\ss^{5}W_{1}\,\tt-3\,\ss^{6}W_{1}\,\tt^{2}-3\,\ss^{4}W_{1}\,{\tt}^{4}+5\,\ss^{4}W_{2}\,\tt^{4}+4\,\ss^{6}W_{2}\,\tt^{2}-2\,\ss^{5}W_2
\,\tt+\ss^{6}W_{3}\,\tt+2\,\ss^{4}W_{3}\,\tt^{3}-4\,\ss^{5}W_{4}\,\tt^{2}+3\,\ss^{6}\tt^{3}W_{4}+\ss^{8}W_{4}\,\tt+3\,\ss^{4}\tt^{5}W_{4}-\tt^{6}W_{1}
\,\ss^{2}+2\,\tt^{6}\ss^{2}W_{2}+\tt^{5}\ss^{2}W_{3}-2
\,\tt^{4}W_{4}\,\ss^{3}+\tt^{7}W_{4}\,\ss^{2}+\ss^{6}W_{1}-\ss^{8}W_{1}+\ss^{8}W_{2}-2\,\ss^{7}W_{4}
$

\noindent In Step 10.5.2. and Step 10.5.3. we get
\[ \Delta_{3}^{*}=\{(\tt,\ss)\in \K^2\,|\, \ss=0\},\,\,\Delta_{4}^{*}=\{(\tt,\ss)\in \K^2\,|\, \ss^2+\tt^2=0\}\cup \Delta_{0}^{*}.\]
Therefore in Step 11. we get
\[ \Delta^*=\{(\tt,\ss)\in \K^2\,|\, \ss=0\}\cup \{(\tt,\ss)\in \K^2\,|\, \tt=0\}\]\[\cup\{(\tt,\ss)\in \K^2\,|\, \ss+\tt=0\}\cup \{(\tt,\ss)\in \K^2\,|\, \ss\pm \ii \tt=0\} \]
Now Step 12. is already  executed, Step 13. is not needed, and in Step 14. we get that all components are rational; indeed lines. Then, for each of the lines we execute Step 15.
\begin{itemize}
\item Let $\cc$ be the line $\ss=0$. We consider the normal proper parametrization $${\cal Q}(\lambda)=(\lambda,0),\,\, \mbox{and} \,\,
{\cal H}(\lambda)=\left(0, \frac{1+\lambda^2}{\lambda}, 1\right). $$
$\cc\cap \Delta_{0}^{*}=\{(0,0)\}\subset \Delta_4$. So, we deal generically with $\cH(\lambda)$, knowing that $\ggg^{\cPw,\cH(\lambda)}=1$. That is, we go back to Step 5 taking ${\cal A}^*$ as $\cH(\lambda)$. We get that $K=1$. We know that $\Delta_{0}^{*}=\emptyset$. In Step 8. we get that the new $\cZ_{\infty}=\{0\}\subset \K$, that $M_1=\lambda(1-\cX_1), M_2=-\cX_2$. Therefore, the new $\Delta_{1}^{*}=\{0\}\subset \K.$ In Step 9.
$R=\cX_{2}^{2}\lambda t^2+\cX_{2}^{2}\lambda-\cX_{2}^{2}t-\cX_{2}^{2}\lambda^2 t-\cX_2-\cX_2\lambda^2+\lambda-\cX_1$ and
$W = \cX_{2}^{2}\lambda$. So $\Delta_{2}^{*}=\Delta_{1}^{*}$. In Step 10.1., we get that $R$ has 4 non-zero coefficients w.r.t. $\{\cX_1,\cX_2\}$. Moreover,
$a=1$ (see Step 10.2.) and $N=1$ (see Step 10.3.). Since $\card(J)=4$ we go through Step 10.5. Then (see Step 10.5.1.), $a_{i_1,j_1}=\lambda$, $L=\lambda$, and $\overline{S}=\lambda^2$. Thus, $\Delta_{3}^{*}=\Delta_{4}^{*}=\Delta_{2}^{*}=\Delta_{1}^{*}$. Summarizing, for all $\lambda\neq 0$ we
$\deg(\Phi_{2,3}(\cH(\lambda))=\deg_{t}(\pp_{\{\cX_1,\cX_2\}}(\res_s(H_1,H_2)))=2$. Therefore, $m_1=2$ and then
\[ \mult((\lambda:0:1))=\mult(\cH(\lambda),\cSw)=\frac{n_1-m_1}{n_2}=2,\,\,\mbox{with $\lambda\neq 0$}. \]
\item Let $\cc$ be the line $\tt=0$. We consider the normal proper parametrization $${\cal Q}(\lambda)=(0,\lambda),\,\, \mbox{and} \,\,
{\cal H}(\lambda)=\left(\lambda, \frac{1+\lambda^2}{\lambda}, 2\right). $$
$\cc\cap \Delta_{0}^{*}=\{(0,0)\}\subset \Delta_4$. So, we deal generically with $\cH(\lambda)$, knowing that $\ggg^{\cPw,\cH(\lambda)}=1$. That is, we go back to Step 5 taking ${\cal A}^*$ as $\cH(\lambda)$. We get that $K=1$. We know that $\Delta_{0}^{*}=\emptyset$. In Step 8. we get that $\cZ_{\infty}= \K$. So, we perform a suitable linear change of parameters in $\cPw$ to avoid that, namely we replace (during the analysis of this curve) $\cPw$ by ${\cal P}_{x_4}(s-t,t+s)$. Then, we get that $\cZ_\infty=\{0\}$. Then, repeating the computation we get that  $\Delta_{3}^{*}=\Delta_{4}^{*}=\Delta_{2}^{*}=\Delta_{1}^{*}=\{0\}$. Summarizing, for all $\lambda\neq 0$ we
$\deg(\Phi_{2,3}(\cH(\lambda))=\deg_{t}(\pp_{\{\cX_1,\cX_2\}}(\res_s(H_1,H_2)))=3$. Therefore, $m_1=3$ and then
\[ \mult((0:\lambda:1))=\mult(\cH(\lambda),\cSw)=\frac{n_1-m_1}{n_2}=1,\,\,\mbox{with $\lambda\neq 0$}. \]
\item The next curve is precisely $\Delta_4$. So, we  postpone its analysis to further levels.
\item  Let $\cc$ be the lines $\tt\pm\ii \ss=0$; we treat both curves simultaneously. We consider the normal proper parametrization $${\cal Q}(\lambda)=(\pm\ii \lambda,\lambda),\,\, \mbox{and} \,\,
{\cal H}(\lambda)=\left( \frac{\lambda^2}{\lambda\pm\ii\lambda}, \frac{1}{\lambda\pm\ii\lambda}, \frac{\pm\ii\lambda+2\lambda}{\lambda\pm\ii\lambda}    \right). $$
$\cc\cap \Delta_{0}^{*}=\{(0,0)\}\subset \Delta_4$. Repeating the computation we get that  $\Delta_{3}^{*}=\Delta_{4}^{*}=\Delta_{2}^{*}=\Delta_{1}^{*}=\{0\}$. Summarizing, for all $\lambda\neq 0$ we
$\deg(\Phi_{2,3}(\cH(\lambda))=\deg_{t}(\pp_{\{\cX_1,\cX_2\}}(\res_s(H_1,H_2)))=3$. Therefore, $m_1=3$ and then
\[ \mult((\pm \ii \lambda:\lambda:1))=\mult(\cH(\lambda),\cSw)=\frac{n_1-m_1}{n_2}=1,\,\,\mbox{with $\lambda\neq 0$}. \]
\end{itemize}
 We go to Step 17. $\Delta_4\cap \Delta_3=\{ (0,0)\}$. So, we work generically with $\Delta_4$ and $\cPz$. So we consider $\cQ(\lambda)=(-\lambda,\lambda)$ and $\cH(\lambda)=\cPz(\cQ(\lambda))$. Proceeding as above, we get $\Delta_{3}^{*}=\Delta_{4}^{*}=\Delta_{2}^{*}=\Delta_{1}^{*}=\{0\}$. Summarizing, for all $\lambda\neq 0$ we
$\deg(\Phi_{2,3}(\cH(\lambda))=\deg_{t}(\pp_{\{\cX_1,\cX_2\}}(\res_s(H_1,H_2)))=3$. Therefore, $m_1=3$ and then
\[ \mult((-\lambda:\lambda:1))=\mult(\cH(\lambda),\cSw)=\frac{n_1-m_1}{n_2}=1,\,\,\mbox{with $\lambda\neq 0$}. \]
In Step 18. since $(\Delta_4\cap \Delta_3)\setminus \Delta_2=\{(0,0)\}$, we compute the multiplicity of ${\cal P}_{x_2}(0,0)=(0,0,0)$ using $\cPy$.
 We get
 \[ \mult((0:0:1))=\mult((0,0,0),\cSy)=3.\]
Since $\Delta_4\cap \Delta_3\cap \Delta_2=\emptyset$ we skip Step 19. and we pass to Step 20.

\noindent In Step 20 we first observe that $\cB=\emptyset$. Moreover, since $\cpp_4(0,1,0)=0$ but $\cpp_1(0,1,0)\neq 0$ we compute $\mult(\tilde{\cP}_{x_1}(0,0),\cSx)$ by applying Method 5. One gets that $m_1:=\deg(\Phi_{2,3}(\tilde{\cP}_{x_1}(0,0))=2$. So,
\[ \mult((0:1:0))=\mult(\tilde{\cP}_{x_1}(0,0),\cSx)=\frac{n_1-m_1}{n_2}=2. \]
In Step 21. we observe that $\cpp_4(1,\lambda,0)=0$ but $\cpp_1(1,\lambda,0)=\lambda^2$. In Step 22. we need to analyze $(1:0:0)$. We do it using
$\hat{\cP}_{x_2}(\hat{s},\hat{v})$ to get
 \[ \mult((1:0:0))=\mult((0,0,0),\cSy)=3.\]
In Step 23. working with $\hat{\cP}_{x_1}(\hat{s},\hat{v})$ we conclude that
\[ \mult((1:\lambda:0))=\mult\left(\left(\frac{1+\lambda^2}{\lambda^2}, 0, 0),\cSx\right)\right)=2\,\,\mbox{for $\lambda\not\in \{0,1,-1\}$} \]
So it only remains to analyze $(1:1:0),(1:-1:0)$. We apply Method 5 with $\cPx$ to get
\[  \mult((1:1:0))=1, \mult((1:-1:0))=1.  \]
In Fig. 1, we summarize the conclusion.

\begin{figure}[h]
\begin{center}
\includegraphics[angle=270,width=9cm]{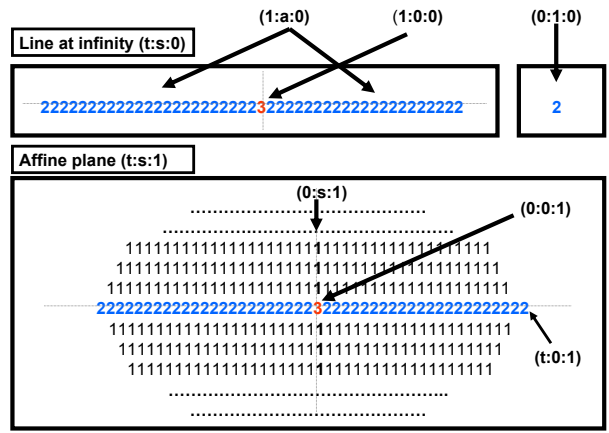}
\caption{Decomposition of the parameter space $\projdos$, where the multiplicities $\mult((t:s:v))$ are represented instead of $(t:s:v)$}\label{dib-fig-1}
\end{center}
\end{figure}


\end{example}

\end{document}